\documentclass[11pt]{amsart}

\usepackage{graphicx}
\usepackage{color}
\usepackage{enumerate,url,amssymb,mathrsfs,upref,pdfsync}
\usepackage[hyperindex,colorlinks,citecolor=blue,pdftex,debug]{hyperref}

\newtheorem{theorem}{Theorem}[section]
\newtheorem{lemma}[theorem]{Lemma}

\newtheorem{proposition}[theorem]{Proposition}
\newtheorem{conjecture}[theorem]{Conjecture}
\newtheorem*{proposition*}{Proposition}

\theoremstyle{definition}
\newtheorem{definition}[theorem]{Definition}
\newtheorem{example}[theorem]{Example}

\newtheorem{corollary}[theorem]{Corollary}

\theoremstyle{remark}
\newtheorem{remark}[theorem]{Remark}

\numberwithin{equation}{section}

\newcommand{\abs}[1]{\lvert#1\rvert}
\newcommand{\norm}[1]{\lVert#1\rVert}

\newcommand{\bow}[1]{\overset{{}_{\bowtie}}{#1}}
\newcommand{\A}{\mathbb{A}}

\newcommand{\C}{\mathbb{C}}

\newcommand{\CC}{C}

\newcommand{\E}{\mathcal{E}}
\newcommand{\EE}{\mathsf{E}}

\newcommand{\R}{\mathbb{R}}

\newcommand{\s}{\mathbb{S}}
\newcommand{\X}{\mathbb{X}}

\newcommand{\D}{\mathfrak{D}}

\newcommand{\Ho}{\mathsf{H}}

\newcommand{\cto}{\xrightarrow[]{{}_{c\delta}}}
\newcommand{\onto}{\overset{{}_{\textnormal{\tiny{onto}}}}{\longrightarrow}}
\DeclareMathOperator{\dist}{dist} \DeclareMathOperator{\Mod}{Mod}
\DeclareMathOperator{\re}{Re} \DeclareMathOperator{\im}{Im}
\DeclareMathOperator{\id}{id}

\def\XXint#1#2#3{{\setbox0=\hbox{$#1{#2#3}{\int}$}
\vcenter{\hbox{$#2#3$}}\kern-.5\wd0}}

\def\le{\leqslant}
\def\ge{\geqslant}

\begin{document}

\title[{Energy-minimal diffeomorphisms between Riemann surfaces}]{Energy-minimal diffeomorphisms between  doubly connected Riemann surfaces}

\author{David Kalaj}
\address{Faculty of Natural Sciences and
Mathematics, University of Montenegro, Cetinjski put b.b. 81000
Podgorica, Montenegro} \email{davidk@ac.me}







\date{\today}


\begin{abstract} Let  $M$ and $N$ be doubly connected
 Riemann surfaces with boundaries and with nonvanishing conformal metrics $\sigma$ and $\rho$ respectively, and assume that $\rho$ is a smooth
metric with bounded Gauss curvature $\mathcal{K}$ and finite area.
The paper establishes the existence of homeomorphisms between $M$
and $N$ that minimize  the Dirichlet energy.
 \smallskip
{\it In the class of all homeomorphisms $f \colon M\onto N$ between
doubly connected Riemann surfaces such that $\Mod M \le \Mod N$
there exists, unique up to conformal authomorphisms of $M$, an
energy-minimal diffeomorphism which is a harmonic diffeomorphism.}
The results improve and extend some recent results of Iwaniec, Koh,
Kovalev and Onninen (Inven. Math. (2011)),
 where the authors
considered bounded doubly connected domains in the complex plane
w.r. to Euclidean metric.
\end{abstract}


\maketitle 
\section{Introduction}\label{intsec}
The primary goal of this paper is to study the minimum of energy
integral of mappings defined between annuli in Riemann surfaces. We
will study the existence of a diffeomorphism $f\colon
\Omega\onto\Omega^*$ of smallest (finite) $\rho-$Dirichlet energy
(see \eqref{hil}) where $\rho$ is an arbitrary smooth metric with
bounded Gauss curvature and finite area defined on $\Omega^*$.  We
extend in this way some results obtained in a recent essential paper
of Iwaniec, Koh, Kovalev and Onninen (\cite{inve}) where the authors
considered the same problem but for planar case. We will follow
closely the ideas and notation of the paper \cite{inve} but we need
to make a new approach, due to the different geometry. We begin by
the definition of harmonic mappings.

\subsection{Harmonic mappings between Riemann surfaces} Let $M$ and $N$ be
Riemann surfaces with metrics $\sigma$ and $\rho$, respectively. If
a mapping $f:M\to N,$ is $C^2$, then $f$ is said to be harmonic (to
avoid the confusion we will sometimes say $\rho$-harmonic) if

\begin{equation}\label{el}
f_{z\overline z}+{(\log \rho^2)}_w\circ f f_z\,f_{\bar z}=0,
\end{equation}
where $z$ and $w$ are the local parameters on $M$ and $N$
respectively. Also $f$ satisfies \eqref{el} if and only if its Hopf
differential
\begin{equation}\label{anal}
\mathrm{Hopf}(f)=\rho^2 \circ f f_z\overline{f_{\bar z}}
\end{equation} is a
holomorphic quadratic differential on $M$. Let $$|\partial
f|^2:=\frac{\rho^2(f(z))}{\sigma^2(z)}\left|\frac{\partial
f}{\partial z}\right|^2\text{ and }|\bar \partial
f|^2:=\frac{\rho^2(f(z))}{\sigma^2(z)}\left|\frac{\partial
f}{\partial \bar z}\right|^2$$ where $\frac{\partial f}{\partial z}$
and $\frac{\partial f}{\partial \bar z}$ are standard complex
partial derivatives. The $\rho-$Jacobian is defined by
$$J(f):=|\partial f|^2-|\partial \bar f|^2.$$ If $u$ is sense preserving,
then the $\rho-$Jacobian is positive. The Hilbert-Schmidt norm of
differential $df$ is the square root of the energy $e(f)$ and is
defined by
\begin{equation}\label{hil}|df|=\sqrt{2|\partial f|^2+2|\partial\bar
f|^2}.\end{equation}  For $g:M \mapsto N$ the $\rho-$
\emph{Dirichlet energy} is defined by
\begin{equation}\label{ener1} \E^\rho[g]=\int_{M} |dg|^2 dV_\sigma,
\end{equation}
where $\partial g$, and $\bar \partial g$ are the partial
derivatives taken with respect to the metrics $\varrho$ and
$\sigma$, and $dV_\sigma$, is the volume element on $(M,\sigma)$,
which in local coordinates takes the form $\sigma^2(w)du\wedge dv,$
$w=u+iv$. Assume that energy integral of $f$ is bounded. Then a
stationary point $f$ of the corresponding functional where the
homotopy class of $f$ is the range of this functional is a harmonic
mapping. The converse is not true. More precisely there exists a
harmonic mapping which is not stationary. The literature is not
always clear on this point because for many authors, a harmonic
mapping is a stationary point of the energy integral. For the last
definition and some important properties of harmonic maps see
\cite{SY}. It follows from the definition that, if $a$ is conformal
and $f$ is harmonic, then $f\circ a$ is harmonic. Moreover if $b$ is
conformal, $b\circ f$ is also harmonic but with respect to
(possibly) an another metric $\rho_1$ (see Lemma~\ref{pali} below).



 Moreover if $N$ and $M$ are double connected
Riemann surfaces with non vanishing metrics $\sigma$ and $\rho$,
then by \cite[Theorem~3.1]{jost}, there exist conformal mappings
$X\colon \Omega \onto M$ and $X^*\colon \Omega^*\onto N$ between
double connected plane domains $\Omega$ and $\Omega^*$ and Riemann
surfaces $M$ and $N$ respectively.

Notice that the harmonicity neither Dirichlet energy do not depend
on metric $\sigma$ on domain so we will assume from now on
$\sigma(z)\equiv 1$. This is why throughout this paper
$M=(\Omega,\mathbf{1})$ and $N=(\Omega^*,\rho)$ will be doubly
connected domains in the complex plane $\C$ (possibly unbounded),
where $\mathbf{1}$ is the Euclidean metric. Moreover $\rho$ is a
nonvanishing smooth metric defined in $\Omega^*$ with bounded Gauss
curvature $\mathcal{K}$ where
\begin{equation}\label{gaus}\mathcal{K}(z)=-\frac{\Delta \log \rho(z)}{\rho^2(z)},\end{equation} (we
put $\kappa:=\sup_{z\in \Omega^*} |\mathcal{K}(z)|<\infty$) and with
finite area defined by
$$\mathcal{A}(\rho)=\int_{\Omega^*}\rho^2(w) du dv, \ \ w=u+iv.$$ We call such a metric
\emph{allowable} one (cf. \cite[P.~11]{Ahb}). If $\rho$ is a given
metric in $\Omega^\ast$, we conventionally extend it to be equal to
$0$ in $\partial\Omega^*$.
As we already pointed out, we will study the minimum of Dirichlet
integral of mappings between certain sets. We refer to introduction
of \cite{inve} and references therein for good setting of this
problem and some connection with the theory of nonlinear elasticity.
Notice first that a change of variables $w=f(z)$ in~\eqref{ener1}
yields
\begin{equation}\label{ener2}
{\E^\rho}[f] = 2\int_{\Omega} \rho^2(f(z))J_f(z)\, dz +
4\int_{\Omega}\rho^2(f(z)) \abs{f_{\bar z}}^2dz\ge 2
\mathcal{A}(\rho)
\end{equation}
where $J_f$ is the Jacobian determinant and $\mathcal{A}(\rho)$ is
the area of $\Omega^*$ and $dz:=dx\wedge dy$ is the area element
w.r. to Lebesgue measure on the complex plane. A conformal mapping
of $f:\Omega\onto\Omega^*$; that is, a homeomorphic solution of the
Cauchy-Riemann system $ f_{\bar z}=0$, would be an obvious choice
for the minimizer of~\eqref{ener2}. For arbitrary multiply connected
domains there is no such mapping.
%
%

In the case of Euclidean
metric $\rho\equiv 1$, the existence of a harmonic diffeomorphism
between certain sets does not imply the existence of an
energy-minimal one, see \cite[Example~9.1]{inve}. Example~9.1 in
\cite{inve} has been constructed with help of affine self-mappings
of the complex plane. For a general metric $\rho$, affine
transformations are not harmonic, thus we cannot produce a similar
example.

First of all,  energy-minimal diffeomorphisms for bounded simply
connected domains exist by virtue of the Riemann mapping theorem. We
continue to study the doubly connected case which has been already
studied before for some special cases. The case of circular annuli
w.r. to Euclidean metric and the metric $\rho(w)=1/|w|$ is
established in \cite{AIM}. This result has been extended to all
radial metrics in \cite{dist}. Recently in \cite{inve}, by using
very interesting approach the authors established the case of
bounded doubly connected domains in the complex plane w.r. to
Euclidean metric, provided that the domain has smaller modulus than
the target \cite[Theorem~1.1]{inve}.
\subsection{Statement of results}
One of the main results of this paper is the following
generalization of the main result in \cite{inve}.
\begin{theorem} \label{mainexist}
Suppose that $\Omega$ and $\Omega^*$ are doubly connected domains in
$\C$ such that $\Mod \Omega\le \Mod\Omega^*$ and let $\rho$ be an
allowable metric in $\Omega^*$. Then there exists an
$\rho-$energy-minimal diffeomorphism $f\colon \Omega\onto \Omega^*$,
which is unique up to a conformal change of variables in $\Omega$.
\end{theorem}

In particular case of the Euclidean metric, our results improve the
main result in \cite{inve} because in contrast to \cite{inve} in
this paper we relax from the assumption that $\Omega$ is a bounded
domain. For the Euclidean metric, the target $\Omega^*$ cannot be
arbitrary, because the energy of every diffeomorphism
$f\colon\Omega\mapsto\Omega^\ast$ is larger than the area
$|\Omega^\ast|$ of $\Omega^*$. In other word the Euclidean metric is
allowable for every domain $\Omega^*$ with finite area (not
necessarily bounded).

The notation $\Mod \Omega$ stands for the {\it conformal modulus} of
$\Omega$. Any doubly connected domain $\Omega\subset\C$, $\Omega\neq
\C\setminus\{a\}$ is conformally equivalent to some \emph{circular
annulus} $$A(\tau):=A(e^{-\tau},1)=\{z \colon e^{-\tau}< \abs{z}<1
\}$$ with $0<\tau \le \infty $. The number $\tau $ defines $\Mod
\Omega:=\tau $. By virtue of \cite[Theorem~3.1]{jost} the modulus
can be defined uniquely for double connected Riemann surfaces with
boundaries as well. The conformal modulus is infinite precisely when
one of components of $\C \setminus \Omega$ degenerates to a point.
We call such domain a {\it punctured domain}.
Theorem~\ref{mainexist} has the following corollary.
 \begin{corollary}
For any  doubly connected domain $\Omega$ and any punctured domain
$\Omega^\ast$ with allowable metric $\rho$ there exists an
energy-minimal diffeomorphism $f\colon \Omega\onto \Omega^*$, which
is unique up to a conformal change of variables in $\Omega$.
\end{corollary}
The main subject of classical Teichm\"uller theory is the existence
of quasiconformal mappings $g\colon \Omega^\ast \onto \Omega$ with
smallest $L^\infty$-norm of the distortion function
\[K_g(w)= \frac{\abs{Dg(w)}^2}{2\, J_g(w)}, \qquad \mbox{ a.e. } w \in
\Omega^\ast\] see \cite{Ahb}. Let $d\mu(w)=\rho^2(w)dudv$ be the
measure at $\Omega^\ast$ and let $L^1=L^1(\Omega^*,\mu)$ be the
Lebesgue space of integrable functions defined in $\Omega^\ast$ (the
norm is defined by
$$\norm{\Phi}_{L^1}:=\int_{\Omega^\ast}\rho^2(w)|\Phi(w)|dudv,$$
where $w=u+iv$). Then the $L^1$-norm of $K_g$ and Dirichlet energy
of the inverse mapping are related via the transformation formula
\begin{equation}\label{identity}
\norm{K_g}_{L^1}= {\E^\rho}[f], \ \ \ \mbox{where}\, f=g^{-1} \colon
\Omega \onto \Omega^\ast.
\end{equation}  For rigorous
statements let us recall that a homeomorphism $g\colon \Omega^\ast
\onto \Omega$ of Sobolev class $W^{1,1}_{loc}(\Omega^\ast)$  has
integrable distortion if
\begin{equation}\label{finitedist}
\abs{Dg(w)}^2 \le 2 K(w)\, J_g(w) \qquad \mbox{ a.e. in }
\Omega^\ast
\end{equation}
for some $K\in L^1(\Omega^\ast,\mu)$. The smallest such $K_g \colon
\Omega^\ast \to [1, \infty)$, denoted by $K_g$, is referred to as
the {distortion function} of $g$.

It turns out that the inverse of any mapping with integrable
distortion has finite Dirichlet energy and the
identity~\eqref{identity} holds. As a consequence of
Theorem~\ref{mainexist} we obtain the following result.
\begin{theorem}\label{thmintdist}
Suppose that $\Omega$ and $\Omega^*$ are doubly connected domains in
$\C$ such that $\Mod \Omega\le \Mod\Omega^*$ and $\rho$ is an
allowable metric in $\Omega^*$. Then, among all homeomorphisms $g
\colon \Omega^\ast \onto \Omega$ there exists, unique up to a
conformal automorphism of $\Omega$, a  mapping of smallest
$L^1=L^1(\Omega^*,\mu)$ norm of the distortion. Here
$d\mu(w)=\rho^2(w)du\wedge dv$.
\end{theorem}

We continue this introduction with a brief outline of the proof of
Theorem~\ref{mainexist} i.e. of its refined variant
Theorem~\ref{q4}. The natural set for our minimization problem is
the subset $W^{1,2}_l$ of local Sobolev space
$W^{1,2}_{loc}(\Omega,\overline{\Omega^\ast})$ of maps $h$ whose
first derivative, i.e. the function $|dh|$ (see \eqref{hil}) lies in
$L^2(\Omega)$.  In this paper functions in the Sobolev spaces are
complex-valued. Let us reserve the notation
$\Ho^{1,2}(\Omega,\Omega^*)$ for the set of all sense-preserving
$W_l^{1,2}$-homeomorphisms $h\colon \Omega\onto\Omega^*$.  When this
set is nonempty, we define
\begin{equation}\label{en2}
\EE^\rho_\Ho(\Omega,\Omega^*)= \inf \{{\E^\rho}[h]\colon h\in
\Ho^{1,2}(\Omega,\Omega^*)\}.
\end{equation}
It is important to note that $W^{1,2}_l$ and the class
$\Ho^{1,2}(\Omega,\Omega^*)$ depend on the metric $\rho$.

A homeomorphism $h\in \Ho^{1,2}(\Omega,\Omega^*)$ is
\emph{energy-minimal} if it attains the infimum in~\eqref{en2}. Let
us notice that the set $\Ho^{1,2}(\Omega,\Omega^*)\subset W^{1,2}_l$
is not bounded neither bounded subsets of
$\Ho^{1,2}(\Omega,\Omega^*)$ are compact.

In order to minimize the energy integral, similarly as in
\cite{inve}, we introduce the class of so-called
\emph{deformations}. These are sense-preserving surjective mappings
of the Sobolev class $W^{1,2}_l$ that can be approximated by
homeomorphisms (diffeomorphisms) in a certain way which make itself
a compact family of mappings. The precise definition of the class of
deformations $\D^\rho(\Omega,\Omega^*)$ is given in \S\ref{stasec}.
Notice that $\Ho^{1,2}(\Omega,\Omega^*)\subset
\D^\rho(\Omega,\Omega^*)$. A deformation is not
necessarily injective neither bounded.  
Define
\begin{equation}\label{en1}
\EE^\rho(\Omega,\Omega^*)= \inf \{{\E^\rho}[h]\colon h\in
\D^\rho(\Omega,\Omega^*)\}
\end{equation}
where ${\E^\rho}[h]$ is as in~\eqref{ener1}. A deformation that
attains the infimum in~\eqref{en1} is called \emph{energy-minimal}.
It is obvious that $\EE^\rho_\Ho(\Omega,\Omega^*)\ge
\EE^\rho(\Omega,\Omega^*)$, but whether the equality holds is not
clear. The following theorem will imply Theorem\ref{mainexist}.
\begin{theorem}\label{q4}
Suppose that $\Omega$ and $\Omega^*$ are doubly connected domains in
$\C$ such that $\Mod \Omega\le \Mod\Omega^*$ and let $\rho$ be an
allowable metric in $\Omega^*$. Then there exists a diffeomorphism
$h\in \Ho^{1,2}(\Omega,\Omega^*)$ that minimizes the energy among
all deformations; that is,
$${\E^\rho}[h]=\EE^\rho(\Omega,\Omega^*)$$ and hence,
$$\EE^\rho_\Ho(\Omega,\Omega^*) = \EE^\rho(\Omega,\Omega^*).$$
%
\end{theorem}
Behind our approach is the calculus of variation.  We rely on
certain inner variations, which yield, according to a trivial
modification of a result of Jost \cite{Job1},  that the Hopf
differential (\S\ref{hopsec}) of an energy-minimal deformation w.r.
to these variation is holomorphic in $\Omega$ and real on its
boundary. Our inner variations will prove efficient. Indeed, we will
show later that the energy-minimal deformation exists and is unique,
up to a conformal automorphisms of $\Omega$, under the condition
$\Mod(\Omega)\le \Mod(\Omega^*)$. By using the result of Jost and
improved Reich-Walczak-type inequalities (\S\ref{reisec}), as in
\cite{inve}, we obtain additional information about the Hopf
differential where the conformal moduli of $\Omega$ and $\Omega^*$
play their roles.

As in \cite{inve}, we consider a one-parameter family of variational
problems in which $\Omega$ changes continuously while $\Omega^\ast$
(and $\rho$) remain fixed. This was crucial idea in \cite{inve} and
works in this setting as well. We establish strict monotonicity of
the minimal energy as a function of the conformal modulus of
$\Omega$ (\S\ref{monsec}) (without additional assumption on the
boundary of target) improving in this way the corresponding result
in \cite{inve}. The proof of Theorem~\ref{q4} (and consequently of
Theorem~\ref{mainexist}), is completed in~\S\ref{exisec}. In
~\S\ref{consec} we establish the strict convexity of the minimal
energy provided that the modulus of domain is smaller than the
modulus of target. In the final section ~\S\ref{K} we consider the
special situation where $\Omega^\ast$ is a circular annulus with
radial metric.  We prove the strict convexity of minimal energy
under the so-called Nitsche condition. In addition we prove that
under the  Nitsche condition the minimal energy is attained for
quasiconformal harmonic mappings which are shown to be natural
generalization of conformal mappings.
\begin{remark} The existence of a harmonic diffeomorphisms  with prescribed boundary data
(which solve the Dirichlet problem)  between certain simply
connected domains in Riemann surfaces was primary purpose of a large
number of papers and books. We refer to the monographs of Jost
\cite{jost} and Hamilton \cite{hamr}. On the other hand, as is
observed by J. C. C. Nitsche \cite{n}, if the domains are doubly
connected (in particular circular annuli), then the existence of a
harmonic diffeomorphism between them is not assured even though we
do not take care of boundary data. In this paper, as a by-product of
main theorem (Theorem~\ref{q4}) it is established the existence of
$\rho-$harmonic diffeomoprhisms between doubly connected domains
$\Omega$ and $\Omega^*$
 in $\C$ provided that $\Mod \Omega\le
\Mod\Omega^*$ and $\rho$ is an allowable metric. Notice that, the
same approach can be deduced for arbitrary doubly connected domains
$\Omega$ and $\Omega^*$ in Riemann surfaces  $(M,\sigma)$ and
$(N,\rho)$. This is a variation of the corresponding result
\cite{IKO4} where the authors settled the same problem but for the
Euclidean metric ($\rho(w)\equiv 1$). Indeed they obtained a better
result in this special case by using the fact that the Euclidean
harmonicity is invariant under affine transformations of the target,
which is not the case for arbitrary $\rho$-harmonic mappings. 

In opposite direction, for every $\epsilon>0$ there exists a pair of
smooth bounded doubly connected domains $\Omega,\Omega^*$ with
$0<\Mod\Omega-\Psi_\rho(\Mod\Omega^*)<\epsilon,$ where
$$\Psi_p=\Psi_\rho(\omega):= \int_{R}^{1}
\frac{\rho(y)dy}{\sqrt{y^2\rho^2(y)-R^2\rho^2(R) }}$$ for which
there is no energy-minimal homeomorphism in
$\Ho^{1,2}(\Omega,\Omega^*)$ (See Section~\ref{K} and the reference
\cite{dist}). \end{remark}

\section{Deformations}\label{stasec}

A homeomorphism of a planar domain is either sense-preserving or
sense-reversing. For  homeomorphisms of the Sobolev class
$W^{1,1}_{\rm loc}(\Omega, \Omega^\ast)$ this implies that the
Jacobian determinant does not change sign: it is either nonnegative
or nonpositive at almost every point~\cite[Theorem 3.3.4]{AIMb}, see
also~\cite{HM}. The homeomorphisms considered in this paper are
sense-preserving.

Let $\Omega$ and $\Omega^*$ be domains in $\C$. Let $a\in \Omega^*$
and $b\in \partial \Omega^*$. We define
$$\dist_\rho(a,b):=\inf_\gamma \int_{\gamma}\rho(w)|dw|,$$ where $\gamma$ ranges over all
rectifiable Jordan arcs connecting $a$ and $b$ within $\Omega^*$ if
the set of such Jordan arcs is not empty (otherwise we
conventionally put $\dist_\rho(a,b)=\infty$). To every mapping $f
\colon \Omega \to \overline{\Omega^\ast}$ we associate a boundary
distance function
$$\delta^\rho_f(z)=\dist_\rho (f(z), \partial
\Omega^\ast)=\inf_{b\in \partial \Omega^*}\dist_\rho(f(z),b)$$ which
is set to $0$ on the boundary of $\Omega$.

We now adapt for our purpose the concepts of $c\delta-$uniform
convergence and of deformation defined for Euclidean metric and
bounded domains in \cite{inve}.
\begin{definition}
A sequence of continuous mappings $h_j\colon \Omega\to {\Omega^*}$
is said to converge \emph{$c\delta$-uniformly} to $h\colon \Omega\to
\overline{\Omega^*}$ if
\begin{enumerate}
\item $h_j\to h$ uniformly on compact subsets of $\Omega$ and
\item $\delta^\rho_{h_j} \to  \delta^\rho_h$ uniformly on $\overline{\Omega}$.
\end{enumerate}
We designate it as $h_j\cto h$. Concerning the item (1) we need to
notice that the Euclidean metric and the nonvanishing smooth metric
$\rho$ are equivalent on compacts of $\Omega^*$.
\end{definition}

\begin{definition}\label{defdef}  A mapping $h\colon \Omega\to\overline{\Omega^*}$
is called a \emph{$\rho$ deformation} (or just deformation) if
\begin{enumerate}
\item\label{dd1}   $h\in W_{loc}^{1,2}$ and $|dh|\in L^2$ (which we write shortly $h\in W_l^{1,2}$);
\item\label{dd2} The Jacobian $J_h:=\det Dh$ is nonnegative  a.e. in $\Omega$;
\item\label{dd2p} $\int_\Omega \rho^2(h(z))J_h  \le \mathcal{A}(\rho)$;
\item\label{dd3} there exist sense-preserving  diffeomorphisms  $h_j\colon \Omega\onto\Omega^*$, called an \emph{approximating sequence},  such that $h_j\cto h$  on $\Omega$.
\end{enumerate}
 The set of $\rho$ deformations
$h\colon \Omega \to \overline{\Omega^\ast}$ is denoted by
$\D^\rho(\Omega,\Omega^*)$.
\end{definition}
Notice first that the condition (1) of the previous definition in
the case of bounded domains $\Omega$ and $\Omega^*$ can be replaced
by  $h\in W^{1,2}$ see Lemma~\ref{pali}.

Notice also that $W^{1,2}_l$ depends on $\rho$ because the
Hilbert-Schmidt norm depends on $\rho$. In order to compare
Definition~\ref{defdef} with the corresponding definition
\cite[Definition~2.2]{inve} prove that, they coincide if $\rho$ is
the Euclidean metric, and $\Omega$ and $\Omega^*$ are doubly
connected bounded  domains. We need to show that, for a given
approximating sequence of homeomorphisms there exists an
approximating sequence of diffeomorphisms of $h_0$. For doubly
connected domains $\Omega$ and $\Omega^*$, there exist conformal
mappings $a:A\onto \Omega$ and $b: \Omega^*\onto A^*$ such that, $A$
and $A^*$ are circular annuli. Now for a given homeomorphisms
$h\colon \Omega\onto\Omega^*$, the mapping $\tilde h=b\circ h\circ
a$ is a homeomorphism between $A$ and $A^*$. We can map $A$ and
$A^*$ be means of diffeomorphisms $\chi$ and $\chi^*$ to $\mathbb S
^{2}\setminus\{N,S\}$, where $N$ and $S$ are north and south pole of
the unit sphere $\mathbb S^2\subset \mathbf R^3$. This induces a
unique homeomorphism $\bow{h} \colon \mathbb S^2 \onto \s^2$. By a
T. Rado's theorem, the homeomorphism $\bow{h}$ can be approximated
uniformly by a sequence of diffeomorphisms $\bow{h}_j$ leaving the
points $N$ and $S$ fixed. This induces a sequence of diffeomorphisms
$h_j\colon \Omega\onto\Omega^*$ converging $c\delta$-uniformly to
$h$. The diagonal selection will produce the desired approximating
sequence of $h_0$.

Further we have that $\Ho^{1,2}(\Omega,\Omega^*)\subset
\D^\rho(\Omega,\Omega^*)$. Outside of some degenerate cases, the set
of deformations is nonempty by Lemma~\ref{laterr} and is  closed
under weak limits in $W^{1,2}(\Omega)$ by Lemma~\ref{wclosed}.

\subsection{Some properties of deformations}\label{defsec} From now
on we assume that $\rho$ is an allowable metric in $\Omega^*$. We
begin with the following lemma which will simplify our approach
throughout the paper.
\begin{lemma}\label{pali}
Assume that $\Omega$ and $\Omega^*$ are doubly connected domains,
and assume that $a\colon \Omega \onto A(\tau)$ and $b\colon
A(\omega) \onto \Omega^*$ are univalent conformal mappings and
define $\rho_1(w)={\rho(b(w))}{|b'(w)|}$, $w\in A(\omega)$. Then
\begin{enumerate}[\ \ \ \ (a)]
\item $\E^{\rho}[b\circ
f\circ a]=\E^{\rho_1}[f]$ provided that one of the two sides exist.
    \item $b\circ f\circ a \in D^\rho(\Omega,\Omega^*)$ if and only if
$f\in D^{\rho_1}(A(\tau),A(\omega)))$.
    \item For Gauss curvature we have $\mathcal{K}_\rho(b(w))=\mathcal{K}_{\rho_1}(w)$.
\item  $\rho$ is an allowable metric if and only if
$\rho_1$ is an allowable metric.
    \item
$W_l^{1,2}(A(\tau),A(\omega))=W^{1,2}(A(\tau),A(\omega)).$
\item $b\circ f\circ a$ is $\rho$-harmonic if and only if $f$ is
$\rho_1-$harmonic.
\end{enumerate}

\end{lemma}
\begin{proof} Let us show (a). By using the change of variables $w=a(z)$ we obtain
\[\begin{split}\E^{\rho_1}[f]&=\int_{A(\tau)}|df|^2=\int_{A(\tau)}\rho_1^2(f(w))|Df|^2\\&=\int_{A(\tau)}\rho^2(b(f(w)))|b'(f(w))|^2|Df(w)|^2
\\&=\int_{\Omega}\rho^2(f(a(z)))|b'(f(a(z)))|^2|Df(a(z))|^2a'(z)\\&=\int_\Omega|d(b\circ
f\circ a)|^2=\E^\rho[b\circ f\circ a].\end{split}\] Prove now (b).
First of all from (a) we have $|d(b\circ f\circ a)|\in L^2$ if and
only if $|d(f)|\in L^2$. Further $J_{b\circ f\circ
a}=|b'|^2J_f|a'|^2$ and therefore $J_{b\circ f\circ a}\ge 0$ if and
only if $J_f\ge 0$. The equality $\ \mathcal{A}(\rho)=
\mathcal{A}(\rho_1)$ follows by using the change of variables
$w=b(z)$ in the integral $\int_{\Omega^\ast}\rho^2(w)dw$.
Furthermore
$$\int_\Omega \rho^2(b\circ f\circ a(z))J_{b\circ f\circ a} = \int_{A(\tau)} \rho_1^2(f(z))J_f.$$
If $h_j$ is the corresponding approximating sequence of
diffeomorphisms for the deformation $f$ then the sequence $b\circ
h_j\circ a$ corresponds to the deformation $b\circ f \circ a$.  The
item (c) follows from formula for Gauss curvature \eqref{gaus}. The
equality $W_l^{1,2}(A(\tau),A(\omega))=W^{1,2}(A(\tau),A(\omega))$
follows from the fact that every mapping $f$ between $A(\tau)$ and
$A(\omega)$ is bounded, and consequently the inequality
$$\|f\|^2_{W^{1,2}}:=\int_{A(\tau)}(|f|^2+|df|^2)<\infty$$ is
automatically satisfied if $\int_{A(\tau)}|df|^2<\infty$. The item
(f) follows from the fact that Hopf differential of $f$ is
holomorphic if and only if Hopf differential of $b\circ f\circ a$ is
holomorphic.
\end{proof}
In \cite[Section~3]{inve} the authors established some essential
properties of the class of deformations
$\D(\Omega,\Omega^\ast):=\D^\rho(\Omega,\Omega^\ast)$ introduced in
Definition~\ref{defdef}, for $\rho\equiv 1$. Two main properties are
that the family $\D^\rho(\Omega,\Omega^\ast)$ is sequentially weakly
closed (Lemma~\ref{exc}) and its members satisfy a change of
variable formula~\eqref{useful}. It must be noticed that all
propositions of \cite[Section~3]{inve} except
\cite[Lemma~3.13]{inve} can be extend as well to the class
$\D^\rho(\Omega,\Omega^\ast)$, with some nonessential changes that
depends on the metric $\rho$; for instance instead of $\int_{\Omega}
J_h\le |\Omega^*|$ use the inequality $\int_{\Omega}
\rho^2(h(z))J_h\le \mathcal{A}(\rho)$ and instead of $\delta_f$
consider $\delta_f^\rho$. Notice also that the degree theory
\cite{Llb} and Lusin's condition for the class $W^{1,2}$ \cite{MaMa}
play important role in the proofs. We will only formulate the
corresponding properties needed in this paper.

\begin{lemma}\label{KM}\cite[Lemma~3.2]{inve}
Let $\Omega$, $\Omega^\ast$ and $\Omega^\circ$ be domains in $\C$.
If $f\colon \Omega^\circ \onto \Omega$ is a quasiconformal mapping
then for any $h\in \D^\rho(\Omega, \Omega^\ast)$ we have $h\circ f
\in \D^\rho(\Omega^\circ, \Omega^\ast)$.
\end{lemma}
Observe that in Lemma~\ref{KM} $\Omega$ and $\Omega^\ast$ need not
be bounded domains. This do not creates problems because $\E[h\circ
f]\le K\E[f]$, where $K$ is the quasiconformality constant of $h$.

\begin{lemma}\label{surjective}\cite[Lemma~3.4]{inve}
For any $h\in \D^\rho(\Omega, \Omega^\ast)$ we have  $h(\Omega)
\supset \Omega^\ast$.
\end{lemma}

\begin{definition}
A continuous mapping $f\colon \mathbb X \to \mathbb Y$ between
metric spaces $\mathbb X$ and $\mathbb Y$  is {\it monotone} if for
each $y\in f(\X)$ the set $f^{-1}(y)$ is compact and connected.
\end{definition}

\begin{proposition}\label{why}~\cite[VIII.2.2]{Wh}
If $\X$ is compact and $f\colon \mathbb X \onto \mathbb Y$ is
monotone then $f^{-1}(C)$ is connected for every connected set $C
\subset \mathbb Y$.
\end{proposition}

\begin{lemma}\label{mono}\cite[Lemma~3.7]{inve}
Let $\Omega$ and $\Omega^\ast$ be doubly connected domains in $\C$,
and  $h\in \D^\rho(\Omega, \Omega^\ast)$.  Then $\bow{h}$ is
monotone.
\end{lemma}
\begin{lemma}\label{multilemma}\cite[Lemma~3.8]{inve}
Let $\Omega$ and $\Omega^\ast$ be domains in $\C$. If $h\in
\D^\rho(\Omega, \Omega^\ast)$, then $h$ satisfies  Lusin's condition
$(N)$ and  $N_\Omega(h, w) =1$ for almost every $w\in \Omega^\ast$.
Also $J_h=0$ almost everywhere in $\Omega \setminus
h^{-1}(\Omega^\ast)$.
\end{lemma}
\begin{corollary}\label{useless}\cite[Corollary~3.9]{inve}
Let $\Omega$ and $\Omega^\ast$ be domains in $\C$. If $h\in
\D^\rho(\Omega, \Omega^\ast)$ and  $v \colon  \overline{\Omega^\ast}
\to [0, \infty)$ is measurable, then
\begin{equation} \label{useful}
\int_\Omega v\big(h(z)\big) J_h(z)\, d z = \int_{\Omega^\ast} v(w)\,
d w.
\end{equation}
\end{corollary}

In general, a deformation may take  a part of $\Omega$ into
$\partial \Omega^\ast$. This is the subject of next lemma.

\begin{lemma}\label{goodset}\cite[Lemma~3.10]{inve}
Suppose that  $h\in \D^\rho(\Omega,\Omega^*)$  where $\Omega$ and
$\Omega^\ast$ are doubly connected domains. Let $G= \{z\in
\Omega \colon h(z) \in \Omega^\ast\}$. Then   $G$ is a domain
separating the boundary components of $\Omega$. Precisely, the two
components of $\partial \Omega$ lie in different components of $\C
\setminus G$.
\end{lemma}

Now we formulate the following extension of
\label{wclosed}\cite[Lemma~3.13]{inve}.
\begin{lemma}\label{exc}
Let $\Omega$ and $\Omega^\ast$ be bounded doubly connected planar
domains. Assume that the boundary components of $\Omega$ do not
degenerate into points. If a sequence $\{h_j\}\subset
\D^\rho(\Omega,\Omega^*)$ converges weakly in $W^{1,2}$, then its
limit belongs to $\D^\rho(\Omega,\Omega^*)$
\end{lemma}
\begin{proof}
The key result needed for the proof of \cite[Lemma~3.13]{inve} was
\cite[Poposition~3.11]{inve}. Due to different metric, instead od
\cite[Poposition~3.11]{inve} we need here a corollary of following
Courant-Lebesgue lemma:
\begin{lemma}\cite[Lemma~1.3.2]{Job1}.\label{mopo} Let $\Omega=A(r,R)$ be a circular annulus.
Let $\Omega^*$ be a bounded doubly connected domain with allowable
metric $\rho$
 and with distance function $\dist_\rho(\cdot,\cdot)$. Let $ 0 < \epsilon <\min\{1,
(R-r)/2\}$ and $z_0\in {\Omega}$ and define $S(z_0,\epsilon):=\{z\in
\Omega\colon |z-z_0|=\epsilon\}$. Suppose $f\in
W^{1,2}(\Omega,\Omega^*)$, with $\E^\rho[f] \le  K$. Then there
exists some $\nu\colon\epsilon < \nu < \sqrt{\epsilon}$ for which
$f_{|S(z_0,\nu)\cap \Omega}$ is absolutely continuous and satisfies
\begin{equation}\label{a35}\dist_\rho(f(z_1),f(z_2))\le (8\pi K)^{1/2}(\log(1/\epsilon))^{-1/2} \end{equation} for all $z_1,z_2\in S(z_0,\nu)\cap \Omega$.
\end{lemma}
\begin{corollary}\label{pol}
Let $f\in \D^\rho(\Omega, \Omega^*)$. Then for
$\dist(z,\partial\Omega)<\sqrt{3}\epsilon/2$  we have $$\dist_\rho
(f(z),
\partial \Omega^\ast)\le (8\pi \E^\rho[f])^{1/2}(\log(1/\epsilon))^{-1/2} .$$
\end{corollary} \begin{proof}[Proof of Corollary~\ref{pol}] Let $K=\E^\rho[f]$. Pick $\epsilon>0$ and let $z^n_0$, $n=1,\dots, n_0$ be a sequence
of points of $\Omega$ such that
$$\partial \Omega\subset \bigcup_{n=1}^{n_0}U(z^n_0,\nu_k/2)\text{   and   } S(z^n_0,\epsilon)\cap
\partial\Omega\neq\emptyset.$$ Here $\nu_k=\nu(z^n_0)\in[\epsilon,\sqrt{\epsilon}]$ is provided by Lemma~\ref{mopo} and by  $U(z,r)$ we denote the open disk in the complex plane with
the center $z$ and the radius $r$. It follows from \eqref{a35} that
$$\dist_\rho (f(z),
\partial \Omega^\ast)\le (8\pi K)^{1/2}(\log(1/\epsilon))^{-1/2} $$
provided that $z\in \Omega$ and $z\in S(z^n_0,\nu_k)$. Then for
$$z\in
\partial (\bigcup_{n=1}^{n_0}U(z^n_0,\nu_k))\cap \Omega$$ we have
$$\dist_\rho (f(z),
\partial \Omega^\ast)\le (8\pi K)^{1/2}(\log(1/\epsilon))^{-1/2} .$$
Since, by Lemma~\ref{mono} $f$ is a monotone mapping, it follows
that
$$\dist_\rho (f(z),
\partial \Omega^\ast)\le (8\pi K)^{1/2}(\log(1/\epsilon))^{-1/2} $$
for $\dist(z,\partial\Omega)<\sqrt{3}\epsilon/2.$ \end{proof} Let
$q$ be a conformal mapping of the annulus $A(\tau)$ onto $\Omega$
and $p$ a conformal mapping of $\Omega^\ast$ onto $A(\omega)$
($0<\tau,\omega<\infty$). Then $f\in D^\rho(\Omega,\Omega^*)$ if and
only if $h=p\circ f\circ q\in D^\rho(A(\tau),A(\omega))$. The rest
of the proof of Lemma~\ref{exc} is similar to that of
\cite[Lemma~3.13]{inve}.
\end{proof}
Since the $\rho-$Dirichlet energy is weak lower semicontinuous (see
for example \cite[Lemma~2.1]{SY}), Lemma~\ref{wclosed} has the
following useful corollary.
\begin{corollary}\label{attain}
Under the hypotheses of Lemma~\ref{exc} there exists $h\in \D^\rho
(\Omega, \Omega^\ast)$ such that ${\E^\rho}[h]= \EE^\rho (\Omega,
\Omega^\ast)$.
\end{corollary}
Finally let us formulate the following property of Sobolev
homeomorphisms.
\begin{lemma}\label{laterr}\cite[Lemma~3.15]{inve}
Let $\Omega$ and $\Omega^*$ be doubly connected domains in
$\C$. Then $\Ho^{1,2}(\Omega,\Omega^*)$ is nonempty, except for one
degenerate case when  $\Mod\Omega=\infty$ and $\Mod\Omega^*<\infty$.
In this case there is no homeomorphism $h\colon \Omega \onto
\Omega^\ast$ of Sobolev class $W^{1,2}$.
\end{lemma}
\section{Harmonic replacement}
Let $\Omega$ be a domain in $\C$ and $U \Subset \Omega$ a simply
connected domain. For a continuous function $f\colon \Omega\to\C$
the continuous function $\mathcal{H}_{U}f \colon \Omega\to\C$ is
called the $\rho-$harmonic modification of $f$, if
$\mathcal{H}_{U}f$ is harmonic in $U$ and agrees with $f$ on
$\Omega\setminus U$. In order to apply "harmonic modification", in
this section we made a small extension of classical
Rado-Kneser-Choquet theorem for the functions which are merely
monotone on the boundary (Lemma~\ref{RKC}).
\begin{definition}
We say that a domain $K\Subset \Omega^*$ with Lipschitz boundary is
\underline{allowable} if $K$ is convex w.r. to the metric $\rho$ and
is contained in a geodesic disk $B_\nu(p)$ with:
\begin{enumerate} \item{} radius $\nu<\pi/(2\kappa)$;
\item{} the cut locus of the center $p$ disjoint from $B_{2\nu}(p)$,
\end{enumerate}
\end{definition}
\begin{lemma}\label{RKC}
Let $U$  be  a simply connected domains in $\C$ and $D$ is an
allowable domain in $\Omega^*$. Suppose that  $f$ is a homeomorphism
from $U$ onto $D$ with continuous extension $f\colon \overline{U}
\to \overline{D}$. Then there exists a unique $\rho-$harmonic
diffeomorphism $h\colon U \onto D $ which agrees with $f$ on the
boundary. In particular,  $h$ has a continuous extension to
$\overline{U}$ which coincides  with $f$ on $\partial U$.
\end{lemma}
\begin{proof}
Since a composition of a $\rho$-harmonic mapping with a conformal
mapping is itself $\rho-$harmonic, by composing by a conformal
mapping of the unit disk onto $U$ we can assume that $U$ is the unit
disk. Assume for a moment that $f\colon \partial U\to \partial D$ is
a homeomorphism. Consider the Dirichlet problem of finding a
$\rho-$harmonic map $h\colon U\rightarrow \Omega^*$ with the given
boundary values: $h|_{\partial U}=f$. By a result of Hildebrandt,
Kaul and Widman \cite{hil} this Dirichlet problem has a solution
contained in $B_\nu(p)$. Moreover by a result of Jost \cite{jj} we
obtain that, since $h\colon\partial U\rightarrow \Omega^*$ is a
homeomorphism onto a Lipschitz convex curve $\partial D$, then the
above solution $h$ is a homeomorphism.

Assume now that  $f\colon \partial U\to \partial D$ is not a
homeomorphism. Since $\partial D$ is Lipschitz, it is a rectifiable
curve with the length $l$. Let $\gamma:[0,l]\to \partial D$ be its
arc-length parametrization and define $\tilde f:[0,2\pi]\to [0,l]$
such that $\gamma(\tilde{f}(t))=f(e^{it})$. By assumptions of the
lemma, we can conclude that $\tilde f$ is monotone. By using
mollifiers we can define a sequence of $C^2$ diffeomorphisms
$\tilde{f_n}:[0,2\pi]\to [0,l]$ converging uniformly to $\tilde f$
(see \cite[p2.~351--352]{ehe} for an explicit construction of the
sequence $\tilde{f_n}$). Moreover we can assume that
$\tilde{f_n}(a_k)=\tilde f(a_k)$, where $a_k\in [0,2\pi)$, $k=1,2,3$
are three different points. Let $h_n$ be a $\rho-$harmonic
diffeomorphism satisfying the boundary condition
$h_n(e^{it})=\gamma(\tilde{f_n}(t))$.

By  \cite[Corollary~7.1]{jost}, on each disc $U(0,r):=\{z\in U\colon
|z|<r\}$, $r < 1$ , there is an a priori bound of the Jacobian
determinant of $h_n(z)$ from below i.e.
\begin{equation}\label{jaja}|J_{h_n}(z)|\ge 1/\delta,\end{equation}
where $\delta=\delta(\kappa,\nu,r,\E^\rho[h|_{U}],|B_{\nu}(p)|)$
($\delta$ also depends on a three point condition of  $h_n$ but this
is satisfied because of the previous consideration). The class of
functions $h_n$ is equicontinuous on the closed unit disk
$\overline{U}$ (see \cite[Lemma~4]{jj} for this argument). By
Arzela-Ascoli theorem we can find a subsequence of $h_n$ converging
uniformly to a mapping $h$. Moreover, by virtue of
\cite[Lemma~5~a)]{jj}, the derivatives of the sequence $h_n$ up to
the second order converge uniformly in compacts to the derivatives
of the mapping $h$. Consequently the limit function $h$ is of class
$C^2$  and satisfies \eqref{jaja}. This means that the Jacobian does
not vanish.
 Therefore $h\in C^2(U)\cap C(\overline{U})$ is a
$\rho-$harmonic function, has the prescribed boundary data $f$ and
its Jacobian is not vanishing  in the interior of $U$. Therefore $h$
is a local diffeomorphic proper mapping and by Banach-Mozur theorem
$h$ is a diffeomorphism in $U$.
\end{proof}
We now apply the harmonic modification to deformations.
\begin{lemma}[Modification Lemma]\label{harmrep}
Let $\Omega$ and $\Omega^*$ be doubly connected domains and assume
that $\rho$ is an allowable metric. Suppose that $h\in
\D^\rho(\Omega,\Omega^*)$
 satisfies $h(\Omega)=\Omega^*$.
Let $D$ be an allowable set such that $\overline{D}\subset
\Omega^*$. Denote $U=h^{-1}(D)$. Then there exists $g=\mathcal{H}_U
h$ and it satisfies the following properties
\begin{enumerate}[\ \ \ (a)]
\item $g \in \D^\rho(\Omega,\Omega^*)$
\item The restriction of $g$ to $U$ is a harmonic diffeomorphism onto $D$.
\item ${\E^\rho}[g]\le {\E^\rho}[h]$ with equality if and only if $g\equiv h$.
\end{enumerate}
\end{lemma}
\begin{proof} The proof is the same as the corresponding result in \cite{inve}, but instead of \cite[Lemma~4.1]{inve} we make use of
Modification Lemma~\ref{RKC} to $h$. Moreover the inequality
${\E^\rho}[\mathcal{H}_U h]\le {\E^\rho}[h]$ follows from the
Dirichlet's principle applied to $\rho-$harmonic mapping $h$.
\end{proof}
\section{Reich-Walczak-type inequalities revisited}\label{reisec}
Here we formulate two important propositions proved in \cite{inve}.
Moreover we improve one of them. It must be said that, the
inequalities do not depend on given metrics on domains. The
inequalities in question are concerned with doubly connected
domains. Similar inequalities were established in~\cite{MM} in the
context of self-homeomorphisms of a disk that agree with the
identity mapping on the boundary.

By following \cite{inve} we introduce notation for several
quantities associated with the derivatives of a mapping $f$. We use
polar coordinates $r$ and $\theta$ and the {\it normal} and {\it
tangential} derivatives
\[f_N=f_r \quad \mbox{ and } \quad f_T= \frac{f_\theta}{r}.\]
In these terms the complex partial derivatives $f_z$ and $f_{\bar
z}$ can be expressed as
\[
f_z = \frac{e^{-i\theta}}{2}\left(f_N - i f_T\right) \text{ and }
\qquad f_{\bar z} = \frac{e^{i\theta}}{2}\left(f_N + i f_T\right).
\]
The Jacobian determinant of $f$ is
\[
J_f= \abs{f_z}^2-\abs{f_{\bar z}}^2 = \im \overline{f_N} f_T.
\]
The {\it normal} and {\it tangential distortion}  of $f$ are defined
as follows.
\begin{align}
K_N^f & := \frac{\abs{f_z+ \frac{\bar z}{z}f_{\bar z}}^2}{J_f}= \frac{\abs{f_N}^2}{ J_f}\\
K_T^f & := \frac{\abs{f_z- \frac{\bar z}{z}f_{\bar z}}^2}{J_f}=
\frac{\abs{f_T}^2} {J_f}
\end{align}
By convention we put $K_N^f=0$ and $K_T^f=0$ if the numerator
vanishes and $K_N^f=\infty$ and $K_T^f=\infty$  if the $J_f$
vanishes but the numerator does not. For a mapping $f\in
W^{1,1}_{\rm loc}$ the quantities $f_N$, $f_T$, and $J_f$ are finite
a.e. Thus  $K_N^f$ and $K_T^f$ are defined a.e. on the domain of
definition of $f$.
\begin{proposition}\label{rwrho}\cite[Proposition~5.1]{inve}
Let $\Omega$ and $\Omega^\ast$ be doubly connected domains such that
$\Omega$ separates $0$ and $\infty$. Suppose that \underline{either}
\begin{enumerate}[(a)]
\item $f\in \mathfrak{D}^\rho(\Omega, \Omega^*)$  \underline{or}
\item $f\colon\Omega \onto \Omega^*$ is a sense-preserving homeomorphism of class $W^{1,1}_{\rm loc}(\Omega, \Omega^*)$.
\end{enumerate}
Then
\begin{equation}\label{rwrho0}
2\pi \Mod\Omega^* \le \int_{\Omega} K_N^f \frac{d z}{\abs{z}^2}.
\end{equation}
\end{proposition}
Unlike \cite[Proposition~5.2]{inve}, our lower bound for the modulus
of the image under a deformation \underline{do not} depends on  the
rectifiability of the boundary of $\Omega^*$. This implies that the
assumption of rectifiability in \cite[Proposition~5.2]{inve}  is
redundant. Namely we have:
\begin{proposition}\label{rwthe}
Let $\A=A(r,R)$ be a circular annulus, $0\le r<R<\infty$, and
$\Omega^*$ a doubly connected domain with finite modulus and suppose
that \underline{either}
\begin{enumerate}[a)]
\item $f\in \mathfrak{D}^\rho(\A, \Omega^*)$  \underline{or}
\item $f\colon\A \onto \Omega^*$ is a sense-preserving homeomorphism of class $W^{1,1}_{\rm loc}(\A, \Omega^*)$.
\end{enumerate}
Then
\begin{equation}\label{rwthe0}
 \int_{\A} K_T^f \, \frac{d z}{\abs{z}^2} \ge 2\pi \frac{(\Mod \A)^2}{\Mod \Omega^\ast}.
\end{equation}
\end{proposition}
\begin{proof} The item b) is proved in \cite[Proposition~5.2]{inve}.
Prove the item a).  Let $\Phi$ be a conformal mapping of $\Omega^*$
onto the annulus $A(r^*,1)$. Then
$$\Mod(\Omega^*)=\Mod(A(r^*,1))=\log\frac{1}{r^*}.$$ Further for
$F=\Phi\circ f$ we have $$J_F(z)=J_\Phi(f(z))\cdot
J_f(z)=|\phi'(f(z))|^2J_f(z)$$ and $$F_T=\Phi'(f(z))f_T.$$ Therefore
$$K_T^F=\frac{|F_T(z)|^2}{J_F(z)}=\frac{|f_T(z)|^2}{J_f(z)}=K_T^f.$$
From Lemma~\ref{KM}, $F\in \mathfrak{D}(\A,A(r^*,1))$.   We conclude
the proof by invoking \cite[Proposition~5.2.~a)]{inve} to the
mapping $F$ and annuli $\A$ and $A(r^*,1)$ which has a rectifiable
boundary.
\end{proof}
\section{Stationary deformations}\label{hopsec}
We call a deformation $h\in \D^\rho(\Omega,\Omega^*)$
\emph{stationary} if
\begin{equation}\label{stat}
\frac{d}{dt}\bigg|_{t=0}{\E^\rho}[h\circ \phi_t^{-1}]=0
\end{equation}
for every family of diffeomorphisms $t\to \phi_t\colon
\Omega\to\Omega$ which depend smoothly on the parameter $t\in\mathbb
R$ and satisfy
$\phi_0=\id$. The latter mean that the mapping $\Omega\times [0,\epsilon_0]\ni (t,z)\to \phi_t(z)\in \Omega $ is a smooth mapping for some $\epsilon_0>0$.  Assume as we may that $\Omega$ is a doubly connected domain with smooth boundary. 
The derivative in~\eqref{stat} exists for any $h\in
W^{1,2}(\Omega)$, see computation in \cite[p.~153-154]{Job1}
(c.f.~\cite[p.~158]{SY}). Every energy-minimal deformation is
stationary. Indeed, $h\circ \phi_t^{-1}$ belongs to
$\D^\rho(\Omega,\Omega^*)$ if $t$ is close to zero by virtue of
Lemma~\ref{KM}, because $\phi_t= \id + o(t)$. The minimal property
of $h$ implies ${\E^\rho}[h\circ \phi_t^{-1}]\ge {\E^\rho}[h]$, from
where we obtain ~\eqref{stat}.

Following verbatim the proof of \cite[Lemma 1.2.2]{Job1} but
beginning by \begin{equation}\label{inst}
\varphi(z)dz^2:=\mathrm{Hopf}(h)=\frac{\rho^2(h(z))}{4}(|h_x|^2-|h_y|^2-2
i \left<h_x,h_y\right>)dz^2
\end{equation}
instead of \cite[Eq.~(1.2.24)]{Job1}
 we obtain the following crucial properties of the stationary mapping in~\eqref{stat}:
\begin{itemize}
\item The function $\varphi:= \rho^2(h(z)) h_z\overline{h_{\bar z}}$, a priori in $L^1(\Omega)$, is holomorphic.
\item If $\partial \Omega$ is $\CC^1$-smooth then $\varphi$ extends continuously to $\overline{\Omega}$, and the quadratic differential $\varphi \, dz^2$  is real on each boundary curve of $\Omega$.
\end{itemize}
Let us consider the particular case $\Omega=A(r,R)$ with
$0<r<R<\infty$. Since $\varphi \, dz^2$ is real on each boundary
circle: $z=p e^{it}$, ($p=r,R$), the differential $\varphi \,
dz^2=-\varphi(z)p^2 e^{2it}dt^2=-z^2 \varphi(z) dt^2$ is real on
$\partial \Omega$. Thus the function $z^2 \varphi(z)$ is real on
$\partial \Omega$, and by the maximum principle to the harmonic
function $\im(z^2 \varphi(z))$ it follows that
\begin{equation}\label{above}
z^2\varphi(z)\equiv c\in\R.
\end{equation}
We now have.
\begin{lemma}\label{ctheory}
Let $\Omega=A(r,R)$ be a circular annulus, $0<r<R<\infty$, and
$\Omega^*$ a doubly connected domain. If $h\in
\D^\rho(\Omega,\Omega^*)$ is a stationary deformation, then
\begin{equation}\label{hopf1}\rho^2(h(z))
h_z\overline{h_{\bar z}} \equiv \frac{c}{z^2}\qquad \text{in }\Omega
\end{equation}
where $c\in\R$ is a constant.  Furthermore,
\begin{equation}\label{important}
\begin{cases}
\abs{h_N}^2 \le J_h, & \quad \mbox{if } c \le 0 \\
\abs{h_T}^2 \le J_h, & \quad \mbox{if } c \ge 0.
\end{cases}
\end{equation}
Finally if $\rho$ is bounded, then
\begin{equation}\label{hilb}|Dh|^2\ge
\frac{4|c|}{R^2\rho_0^2},\end{equation} where
$\rho_0=\sup_{w\in\Omega^*}\rho(w)$.
\end{lemma}
\begin{proof} Proof goes along the lines of the proof of \cite[Lemma~6.1]{inve} but since we need some relations
of the proof we include the proof here. The relation ~\eqref{hopf1}
with some $c\in\R$ was already established in~\eqref{above}.
Separating the real and imaginary parts in~\eqref{hopf1} we arrive
at two equations
\begin{align}
\rho^2(h(z))(\abs{h_N}^2 - \abs{h_T}^2) &= \frac{4c}{\abs{z}^2}; \label{hopf2a}\\
\rho^2(h(z))\re (\overline{h_N} h_T) &=0. \label{hopf2b}
\end{align}
Recall that $J_h = \im \overline{h_N} h_T\ge 0$, which in view
of~\eqref{hopf2b} reads as
\begin{equation}\label{hopf5}
J_h = \abs{h_N} \abs{h_T}.
\end{equation}
Combining this with~\eqref{hopf2a} the claim~\eqref{important}
follows. The relation \eqref{hilb} follows from \eqref{hopf1} and
the fact that $|h_z|\ge |h_{\bar z}|$.
\end{proof}
Lemma~\ref{ctheory} together with Propositions~\ref{rwrho}
and~\ref{rwthe} give the following improvement of
\cite[Corollary~6.2]{inve}.
\begin{corollary}\label{cpositive} Under the hypotheses of Lemma~\ref{ctheory}, we have
\begin{itemize}
\item if $\Mod\Omega<\Mod\Omega^*$, then $c>0$
\item if $\Mod\Omega>\Mod\Omega^*$, then $c<0$.
\end{itemize}
\end{corollary}
\section{Monotonicity of minimum energy function}\label{monsec}
Because the Dirichlet integral and the class of deformations are
conformally invariant (Lemma~\ref{KM}), the minimal energy
$\EE^\rho(\Omega,\Omega^*)$, defined by~\eqref{en1}, depends only on
the conformal type of $\Omega$ provided that $\Omega^\ast$ is fixed.
Moreover by Lemma~\ref{pali}, we can assume that $\Omega^*$ is a
circular annulus $A(r,1)$.  This leads us to consider a
one-parameter family of extremal problems for homeomorphisms
$A(\tau)\onto A(\omega)$. In this section consider  the quantity
$\EE^\rho(\tau,\omega):=\EE^\rho(A(\tau),A(\omega))$ as a function
of $\tau$, called the {\it minimum energy function}. Notice that
$A(\tau)$ is conformally equivalent to $A(1,e^{\tau})$ and the
latter has been used in \cite{inve}, however for some technical
reasons, see Section~\ref{K}, $A(\tau)$ has been shown to be more
appropriate. It is clear that the function $\EE^\rho(\tau,\omega)$
attains its minimum at $\tau=\omega$. Indeed, by~\eqref{ener2} for
every $\tau$ we have $\EE^\rho (\tau, \omega) \ge 2
\mathcal{A}(\rho)$, with equality if and only if $\tau=\omega$. The
following monotonicity result, which extends this observation, will
be very important in the proof of Theorem~\ref{q4}. It extends and
improves the corresponding \cite[Proposition~7.1]{inve}.
\begin{proposition}\label{q3}
Let $\omega>0$ and $\rho$ be a smooth metric with bounded Gauss
curvature in $A(\omega)$. The function $\tau\mapsto
\EE^\rho(\tau,\omega)$ is strictly decreasing for $0<\tau<\omega$
and  strictly increasing for $\tau>\omega$. Furthermore
\[\frac{d}{dt}\bigg|_{\tau=\tau_\circ} \EE^\rho (\tau, \omega)= -8
\pi c,\] where $c$ is defined in Lemma~\ref{ctheory} and the
constant $c$ depends only on $\tau$ and $\omega$.
\end{proposition}
Similarly as in \cite{inve}, the proof of Proposition~\ref{q3}
requires auxiliary results concerning the normal and tangential
energies
\[\E^\rho_N[h]=\int_{\Omega}\rho^2 \abs{h_N}^2,\qquad  \E^\rho_T[h]=\int_{\Omega}\rho^2\abs{h_T}^2.\]
First of all ${\E^\rho}[h]=\E^\rho_N[h]+\E^\rho_T[h]$.
Both functionals $\E^\rho_N[h]$ and $\E^\rho_T[h]$ transform in a
straightforward way under composition with the power stretch mapping
\begin{equation}\label{stret0}
\psi(z):=\abs{z}^{\alpha-1}z, \qquad 0<\alpha <\infty.
\end{equation}
By using the formula $\det(D\psi(z))=\alpha |z|^{2\alpha-2}$, we
obtain
\begin{equation}\label{stret}
\E^\rho_N[h\circ \psi]= {\alpha}\, \E^\rho_N[h], \qquad
\E^\rho_T[h\circ \psi]= \frac{1}{\alpha} \,\E^\rho_T[h].
\end{equation}
As in \cite{inve}, the domain of definition of $h$ here is
irrelevant because the computation is local.
\begin{lemma}\label{majorant}
Let $\omega\in(0,\infty)$ and  $\tau_\circ\in (0,\infty)$. Suppose
that $h^\circ\in \D^\rho(A(\tau_\circ),A(\omega))$ is an
energy-minimal deformation. Then for all $0<\tau<\infty$ we have
\begin{equation}
\EE^\rho(\tau,\omega) \le
\frac{\tau_\circ}{\tau}\,\E^\rho_N[h^\circ] +
\frac{\tau}{\tau_\circ}\,\E^\rho_T[h^\circ].
\end{equation}
\end{lemma}
\begin{proof} Proof goes along the lines of the proof of \cite[Lemma~7.2]{inve}.
\end{proof}
Let us apply Lemma~\ref{majorant} with $\tau_\circ = \omega$. In
this case $h^\circ \colon \Omega \onto \Omega^\ast$ is conformal so
$\E^\rho_N[h^\circ]=\E^\rho_T[h^\circ]=\mathcal{A}(\rho)$. We obtain
the following simple upper bound for the  minimal energy function,
\begin{equation}\label{upper}
\EE^\rho(\tau,\omega) \le \left(\frac{\omega}{\tau} +
\frac{\tau}{\omega}\right) \mathcal{A}(\rho), \quad 0< \tau <
\infty.
\end{equation}
\begin{corollary}\label{econt}
The function $\EE^\rho(\tau,\omega) $ is locally Lipschitz for
$0<\tau<\infty$.
\end{corollary}
Indeed the existence of $h^\circ$ in Lemma~\ref{majorant}  is
assured by Corollary~\ref{attain}. From Lemma~\ref{majorant} for
arbitrary $0 < \tau_\circ, \tau < \infty$ we have
\begin{equation}\label{econt1}
\begin{split}
\EE^\rho(\tau,\omega) -\EE^\rho(\tau_\circ,\omega)
& \le  \frac{\tau_\circ}{\tau} \E^\rho_N[h^\circ] + \frac{\tau}{\tau_\circ} \E^\rho_T[h^\circ] -  \E^\rho_N[h^\circ] - \E^\rho_T[h^\circ]\\
&=
(\tau-\tau_\circ)\left\{\frac{\E^\rho_T[h^\circ]}{\tau_\circ}-\frac{\E^\rho_N[h^\circ]}{\tau}.\right\}
\end{split}
\end{equation}
From this we obtain the local Lipschitz property.
\begin{proof}[Proof of Proposition~\ref{q3}] 
Since $\EE^\rho(\tau,\omega)$ is locally Lipschitz, its derivative
exists for almost every $\tau \in(0,\infty)$. Fix such a point of
differentiability, say  $0<\tau_\circ<\Mod\Omega^*$ (the other
possibility is $\tau_\circ>\Mod \Omega^*$). Then as in
\cite[Proposition~7.2]{inve} we obtain
\[\frac{d}{dt}\bigg|_{\tau=\tau_\circ} \EE^\rho (\tau, \omega)= -8 \pi c.\]
Corollary~\ref{cpositive} completes the proof. This shows that $c$
depends only on $\tau_\circ$, $\omega$ and $\rho$ but not on
$h^\circ$.
\end{proof}
\section{Proof of Theorem~\ref{q4} and Theorem~\ref{thmintdist}}\label{exisec}
In order to prove our main result we need the following extension of
\cite[Proposition~8.1]{inve}.
\begin{proposition}\label{gtheory}
Let $\Omega$ and $\Omega^*$ be doubly connected domains. Suppose
that $h\in \D^\rho(\Omega,\Omega^*)$ satisfies
${\E^\rho}[h]=\EE^\rho(\Omega,\Omega^*)$. Let $G=\{z\in\Omega\colon
h(z)\in\Omega^*\}$. Then $G$ is a doubly connected domain which
separates the boundary components of $\Omega$. The restriction of
$h$ to $G$ is a $\rho-$harmonic diffeomorphism onto $\Omega^*$.
\end{proposition}
\begin{proof} The fact that $G$ is a domain that separates the boundary components of $\Omega$ was established in Lemma~\ref{goodset}.
For every point $z\in G$ there exists a neighborhood $D=D_x$ in
which $h$ is a harmonic diffeomorphism. Indeed, otherwise we would
be able to find a $\rho$ deformation with strictly smaller energy by
means of Lemma~\ref{harmrep}. Namely for $w=h(z)\in \Omega^*$ we can
find an allowable neighborhood $D\subset \Omega^*$. Then by taking
$U=h^{-1}(D)$ and making  use of Lemma~\ref{harmrep} for $G$ instead
od $\Omega$ we obtain the previous fact. Thus, $h\colon
G\onto\Omega^\ast$ is a local homeomorphism. By an extension of Lewy
theorem due to Heinz, \cite[Theorem~11]{EH}, $h$ is a local
diffeomorphism. On the other hand, for each $w\in \Omega^\ast$ the
preimage $h^{-1}(w)$ is connected by Lemma~\ref{mono}. It follows
that $h\colon G\onto\Omega^*$ is a diffeomorphism. Being a
diffeomorphic image of $\Omega^*$, the domain $G$ must be doubly
connected.
\end{proof}
\begin{proof}[Proof of Theorem~\ref{q4}]
 If $\Mod\Omega=\Mod\Omega^*$, then the domains are conformally
equivalent. As observed in \S\ref{intsec}, a conformal mapping
minimizes the $\rho-$Dirichlet energy. Thus we only need to consider
the case $\Mod\Omega<\Mod\Omega^*$. Assume that $a\colon \Omega\onto
A(\tau) $ and $b\colon A(\omega)\onto \Omega^*$ are conformal
mappings.

Let $h$ and $G$ be as in Proposition~\ref{gtheory}. The existence of
such $h$ is guaranteed by Corollary \ref{attain}. Since $G$
separates the boundary components of $\Omega$, we have $\Mod G\le
\Mod A(\tau)=\tau$ with equality if and only if
$G=\Omega$~\cite[Lemma~6.3]{LVb}.
 If $\Mod G< \tau$, then by Proposition~\ref{q3}
\[
\int_G \rho^2\circ h\abs{Dh}^2 \ge \EE^\rho(\Mod G,\Omega^*) >
\EE^\rho(\tau,\Omega^*) = \int_{A(\tau)} \rho^2\circ h \abs{Dh}^2
\]
which is absurd because $G\subset A(\tau)$. Thus $G=A(\tau)$. By
Proposition~\ref{gtheory} the mapping $h\colon A(\tau)\to
A(\omega))$ is a harmonic diffeomorphism. The uniqueness statement
will follow from Proposition~\ref{needed}. Then, by Lemma~\ref{pali}
$h$ is a minimizer of $\rho_1:=\rho\circ b \cdot |b'|-$energy
between $A(\tau) $ and $A(\omega)$. Let $f^\circ = b\circ h\circ a$.
Then $f^\circ\colon \Omega\onto\Omega^\ast$ is a $\rho$-harmonic
mapping. Moreover by Lemma~\ref{pali}, $f^\circ$ is a minimizer
because $h$ is a minimizer. Indeed by Lemma~\ref{pali} $g\in
\D^\rho(A(\tau),A(\omega))$ if and only if $f=b\circ g\circ a\in
\D^\rho(\Omega,\Omega^*)$ and $\E^{\rho}[b\circ g\circ
a]=\E^{\rho_1}[g]$. Thus
\[\begin{split}\E^{\rho_1}[h]&=\min\{\E^{\rho_1}[g]\colon g\in
\D^\rho(A(\tau),A(\omega))\}\\&=\min\{\E^{\rho}[b\circ g\circ a]:
g\in \D^\rho(A(\tau),A(\omega))\}\\&=\min\{\E^{\rho}[f]\colon f\in
\D^\rho(\Omega,\Omega^*)\}=\E^{\rho}[f^\circ].\end{split}\]
\end{proof}
\begin{proof}[Proof of Theorem~\ref{thmintdist}]
Suppose $\Mod \Omega \le \Mod \Omega^\ast$ and let $f_\circ \colon
\Omega \onto \Omega^\ast$ be an energy-minimal diffeomorphism
provided to us by Theorem~\ref{q4}. Assume as we may also that
$\Omega=A(r,1)$. For every homeomorphism $g \colon \Omega^\ast \onto
\Omega$ with $L^1=L^1(\Omega^*,d\mu)$ integrable distortion the
inverse map $f=g^{-1} \colon \Omega \onto \Omega^\ast$ belongs to
the Sobolev class $W^{1,2}_{loc}(\Omega)$ (because $\rho$ is smooth
in $\Omega$). Let $\Omega_n=A(r+1/n,1-1/n)$ be an exhaustion by
compact sets of $\Omega=A(r,1)$ and let $\Omega^*_n=f(\Omega_n)$.
Then from Corollary~\ref{iva} below and ~\cite{HK} we have
\begin{equation}\label{blah}
\begin{split}
\int_{\Omega^\ast} \rho^2(w) K_g(w)\, d w &= \lim_{n\to \infty}
\int_{\Omega^\ast_n} \rho^2(w) K_g(w)\, d w\\&=\lim_{n\to
\infty}\int_{\Omega_n} \rho^2(f(z))\abs{Df(z)}^2 \, dz\\&\ge
\lim_{n\to \infty}\int_{\Omega_n}
\rho^2(f_\circ(z))\abs{Df_\circ(z)}^2 \, dz\\&= \int_\Omega
\rho^2(f(z)) \abs{Df_\circ (z)}^2\, dz\\&= \int_{\Omega^\ast}
\rho^2(w) K_{g_\circ} (w)\, dw
\end{split}
\end{equation}
where $g_\circ=f_\circ^{-1}$.  The latter follows by change of
variables in integral, where the substitution function is the
$C^\infty$-smooth diffeomorphism $g_\circ$.  If equality holds
in~\eqref{blah} then, by Theorem~\ref{q4}, the mapping $f_\circ^{-1}
\circ f$ is conformal.
\end{proof}
\section{Convexity of the minimum energy function}\label{consec}
In \S\ref{monsec} we proved that for every doubly connected domain
$\Omega^*$ the function $\EE^\rho(\tau,\omega)$ is decreasing for
$0<\tau<\omega$ and increasing for $\tau>\omega$. The minimum of
this function is attained at $\tau=\omega$, i.e., in the case of
conformal equivalence. In this section we prove:

\begin{theorem}\label{convextheorem}
Let $\omega>0$. The function $\tau\mapsto \EE^\rho(\tau,\omega)$ is
strictly convex for $0<\tau<\omega$.
\end{theorem}
Theorem~\ref{convextheorem}  is an extension of corresponding
\cite[Theorem~10.1]{inve}, which is, according to Lemma~\ref{pali} a
special case of Theorem~\ref{convextheorem} for $\rho$ being a
modulus of an analytic function. Before proving we establish some
propositions which we believe can be of interest in its own right.
On the other hand the item (2) of Proposition~\ref{needed}
establishes the uniqueness part of Theorem~\ref{mainexist}.
\begin{proposition}\label{needed}
Let $\omega>0$. Suppose that $h\in \D^\rho(A(\tau),A(\omega))$ and
${h^\circ}\in \D^\rho(A(\tau_\circ),A(\omega))$ are energy-minimal
deformations. In particular, by Lemma~\ref{ctheory},
\begin{equation}\label{hopf11}
\rho^2(h(z))h_z\overline{h_{\bar z}} \equiv \frac{c}{z^2}\qquad
\text{in }A(\tau)
\end{equation}
and
\begin{equation}\label{hopf12}
\rho^2({h^\circ}(z))h^\circ_z\overline{h^\circ_{\bar z}} \equiv
\frac{c_\circ}{z^2}\qquad \text{in }A(\tau_\circ).
\end{equation}
Let $G^\circ=(h^\circ)^{-1}(A(\omega))$ and $G = h^{-1}(A(\omega))$.
Then
\begin{equation}\label{claim10}
{\E^\rho}[{h^\circ}]-{\E^\rho}[h]   \ge
4|c_\circ|\int_{A(\tau_\circ)\setminus G^\circ}\frac{dz}{|z|^2}+8\pi
c (\tau-\Mod(G^\circ))\end{equation} and
\begin{equation}\label{claim11} {\E^\rho}[h]-{\E^\rho}[{h^\circ}]   \ge4|c|\int_{A(\tau)\setminus G}\frac{dz}{|z|^2}+ 8\pi
c_\circ (\tau_\circ-\Mod(G)).\end{equation} Moreover:
\\
(1) Equalities hold in both \eqref{claim10} and \eqref{claim11} if
and only if $G^\circ$ and $G$ are circular annuli and
${(h^\circ)}^{-1}\circ h $ is a conformal mapping between $G$ and
$G^\circ$.
\\
(2) If ${h^\circ}$ and $h$ are minimizers of $\E^\rho$ over the
domain $A(\tau)$, then
$$h\big|_{G^\circ}={h^\circ}\big|_{G}\circ f,$$ where $f\colon G \onto
G^\circ$ is a conformal mapping and
$$\int_{G}\frac{dz}{|z|^2}=\int_{G^\circ}\frac{dz}{|z|^2}.$$ In particular
$\Mod(G)=\Mod(G^\circ)$.
\end{proposition}
\begin{proof} We will only sketch the proof of \eqref{hopf11} (\eqref{hopf12}), since the corresponding proof of \cite[Proposition~10.2]{inve} applies for
our proof.
 First we consider the easy case $c=0$. In this case $h$ is conformal, which implies
${\E^\rho}[h]=2\mathcal{A}(\rho)$. On the other hand,
${\E^\rho}[{h^\circ}]\ge 2\mathcal{A}(\rho)$ with equality if and
only if ${h^\circ}$ is conformal, see~\eqref{ener2}.

Let $G:=\{z\in A(\tau)\colon h(z)\in A(\omega)\}$ and
$G^\circ:=\{z\in A(\tau_\circ)\colon {h^\circ}(z)\in A(\omega)\}$.
Then, by Proposition~\ref{gtheory} $G$ and $G^\circ$ are doubly
connected domains that separates the boundary components of
$A(\tau)$ and $A(\tau_\circ)$ respectively. Moreover \[{h^\circ}
\colon G^\circ\onto A(\omega)\text{\, and \, }h \colon G \onto
A(\omega).\] It remains to deal with $c\ne 0$. The composition
\[
f= (h^\circ)^{-1} \circ h \colon A(\tau) \onto G^\circ
\]
lies in $W^{1,2}_{\rm loc} (A(\tau))$ and is not homotopic to a
constant mapping.  Moreover, the restriction of $f$ to the domain
$G$ is a diffeomorphism onto $G^\circ$, by virtue of
Proposition~\ref{gtheory}. Thus, $f$ possesses a right inverse
$f^{-1}\colon G^\circ\onto G$ which is also a diffeomorphism. Now we
estimate ${\E^\rho}[{h^\circ}]-{\E^\rho}[h]$ by using the change of
variables $w=f(z)$ and proceeding similarly as in
\cite[Proposition~10.2]{inve}, but using this time our improved
Proposition~\ref{rwthe} (b) and Lemma~\ref{multilemma}. Let us only
state one of key sequences of relations for the proof.
\begin{align}
{\E^\rho}[{h^\circ}]&-4|c_\circ|\int_{A(\tau_\circ)\setminus
G^\circ}\frac{dz}{|z|^2}-\int_{G} \rho^2(h(z)) \abs{Dh}^2\nonumber
\\&\label{cchain}= 4\int_{G}
 \rho^2(h(z)) \frac{\left(\abs{h_z}^2+\abs{h_{\bar z}}^2\right)\, \abs{f_{\bar
z}}^2
-2\re \left[h_z \overline{h_{\bar z}} \overline{f_z}f_{\bar z}\right]}{J_f}\, d z  \\
&  \ge 4 \int_{G}  \rho^2(h(z)) \frac{ 2\abs{h_z h_{\bar z}}\,
\abs{f_{\bar z}}^2
-2\re \left[h_z\overline{h_{\bar z}} \overline{f_z}f_{\bar z}\right]}{J_f}\, d z \nonumber \\
& = 4\abs{c} \int_{G} \left[\frac{\abs{f_z-\sigma f_{\bar
z}}^2}{J_f} -1 \right] \, \frac{dz}{\abs{z}^2}, \quad \mbox{where }
\sigma=\sigma(z)= \frac{c\bar z}{\abs{c}z}. \nonumber
\end{align} \emph{Proof
of statement (1)}.  Since $h$ is a sense-preserving diffeomorphism
in $G^\circ$, we have $\abs{h_z}>\abs{h_{\bar z}}$ everywhere in
$G^\circ$ . If equality holds in~\eqref{claim10}, then it also holds
in~\eqref{cchain}. The latter is only possible if $f_{\bar z}\equiv
0$ in $G$. Thus $f\colon G\onto G^\circ $ is a conformal mapping.
This implies $\Mod(G)=\Mod(G^\circ)$. Moreover, by adding
\eqref{claim10} and \eqref{claim11} we have
$$\int_{G^{\circ}}\frac{dz}{|z|^2}=2\pi \Mod(G^{\circ})$$ and $$\int_{G}\frac{dz}{|z|^2}=2\pi
\Mod(G).$$ To continue we need the following lemma
\begin{lemma}\label{lemo}For a doubly connected domain $G$ separating $0$ and $\infty$ we have the
inequality \begin{equation}\label{qeq}\int_{G}\frac{dz}{|z|^2}\ge
2\pi \Mod(G).\end{equation} The equality is attained if and only if
$G$ is a circular annulus.
\end{lemma}
\begin{proof}[Proof of Lemma~\ref{lemo}] The inequality statement can be deduced for example from \cite[Proposition~5.1]{inve}.
In order to prove the equality statement we should include an
inequalities of \cite[Proposition~5.1]{inve} and study it more
closely. One of inequalities in \cite[Proposition~5.1]{inve}
implying \eqref{qeq} is
\begin{equation}\label{konko}
\begin{split}
\left( \int_G \frac{\abs{a'}}{\abs{a}}\, \frac{dz}{\abs{z}}
\right)^2 &\le \int_G \frac{|a'|^2}{\abs{a}^2} \; \int_G \frac{d
z}{\abs{z}^2},
\end{split}
\end{equation}
where $a$ is a conformal mapping of $G$ onto an annulus
$A(r_\ast,1)$. 
The equality is attained in Cauchy-Schwarz inequality \eqref{konko}
if and only if $$\frac{|a'|}{|a|}=\frac{\alpha}{|z|}.$$ Thus
$a(z)=\beta z$ for some constant $\beta\neq 0$, implying that $G$ is
a circular annulus.
\end{proof}
From Lemma~\ref{lemo} we deduce that $G$ and $G^\circ$ are circular
annuli.

\emph{Proof of statement (2)}. Assume that ${h^\circ}\in
\D^\rho(A(\tau_\circ),A(\omega))$ and $h\in
\D^\rho(A(\tau_\circ),A(\omega))$ are the minimizers of the energy
$\E^\rho$. Then $\E^\rho[{h^\circ}]=\E^\rho[h]$. Since $c$ depends
only on $\tau$, it
follows that $c_\circ=c$. 
Let $G$ and $G^\circ$ and $f$ be defined as above. Without loos of
generality we can assume that
\begin{equation}\label{wlg}\int_G\frac{dz}{|z|^2}\ge
\int_{G^\circ}\frac{dz}{|z|^2}.\end{equation} From \eqref{cchain},
we have
\begin{align}
{\E^\rho}[{h^\circ}]&-\E^\rho[h]-4|c|\int_{A(\tau_\circ)\setminus
G^\circ}\frac{dz}{|z|^2}+4|c|\int_{A(\tau_\circ)\setminus
G}\frac{dz}{|z|^2}\nonumber
\\&= 4\int_{G}
 \rho^2(h(z)) \frac{\left(\abs{h_z}^2+\abs{h_{\bar z}}^2\right)\, \abs{f_{\bar
z}}^2 -2\re \left[h_z \overline{h_{\bar z}} \overline{f_z}f_{\bar
z}\right]}{J_f}\, d z\nonumber.
\end{align}
Assume that
$$\E^\rho[{h^\circ}]=\E^\rho[h].$$  Then \begin{equation}\label{qe}\begin{split}4\int_{G}
 \rho^2(h(z)) &\frac{\left(\abs{h_z}^2+\abs{h_{\bar z}}^2\right)\, \abs{f_{\bar
z}}^2 -2\re \left[h_z \overline{h_{\bar z}} \overline{f_z}f_{\bar
z}\right]}{J_f}\\&+8|c|\pi\left(\int_G\frac{dz}{|z|^2}-
\int_{G^\circ}\frac{dz}{|z|^2}\right)=0.\end{split}\end{equation}
This implies that $f_{\bar z}\equiv 0$, i.e. $f$ is a conformal
mapping and therefore $\Mod(G)=\Mod(G^\circ)$ and
$$\int_G\frac{dz}{|z|^2}- \int_{G^\circ}\frac{dz}{|z|^2}=0.$$ This
finishes the proof of Proposition~\ref{needed}.
\end{proof}
By following the proof of Proposition~\ref{needed} we obtain the
following useful variation of \cite[Proposition~10.2]{inve}. 
\begin{proposition}\label{nee}
Let $\Omega^\ast$ be a doubly connected domain. Suppose that $h\in
\D^\rho(A(\tau_\circ),\Omega^\ast)$ is a diffeomorphic
deformation with Hopf differential $c dz^2/z^2$. 
Then for any diffeomorphism $g\colon A(\tau) \to \Omega^\ast$ we
have
\begin{equation}\label{claim100}
\E^\rho[g]-\E^\rho[h] \ge  8\pi c (\tau_\circ-\tau).
\end{equation}
The equality holds in~\eqref{claim100} if and only if
$\tau=\tau_\circ$ and $g^{-1} \circ h$ is a conformal mapping of
$A(\tau_\circ)$ onto itself.
\end{proposition}
Now we can prove the following important corollary
\begin{corollary}\label{iva} If $h\in \D^\rho(A(\tau_\circ),\Omega^\ast)$ is a diffeomorphic
deformation with Hopf differential $c dz^2/z^2$ then $h$ is a
minimizer (under the class of deformation) and it is unique up to
reparametrization by a conformal mapping. In particular the
restriction of every such deformation in a circular annulus
$A'\subset A(\tau_\circ)$ is a minimizer.\end{corollary}
\begin{proof}[Proof of Corollary~\ref{iva}] Assume that $f$ is a deformation. Let  $f_j$ be a \emph{approximating sequence} of
diffeomorphisms $f_j\colon \Omega\onto\Omega^*$ such that $f_j\cto
f$ on $\Omega$. Since $\E^\rho$ is weak lower semicontinuous (see
\cite[Lemma~2.1]{SY}) it follows that
$$\underset{j\to\infty}{\lim\,\sup}\,\E^\rho[f_j]\le E^\rho[f].$$ On
the other hand from \eqref{claim10} by taking $\tau_\circ=\tau$ we
have $\E^\rho[h]\le \E^\rho[f_j]$. Therefore $\E^\rho[h]\le
\E^\rho[f]$.
\end{proof}
\begin{proof}[Proof of Theorem~\ref{convextheorem}]
Proof goes along the lines of the proof of Lemma
\cite[Theorem~10.1]{inve} by using Proposition~\ref{nee}. 
\end{proof}
\begin{corollary}\label{newa}
Let $\mathcal{T}=\{\tau: \EE^\rho(\tau,\omega)=\E^\rho[g]\text{ for some
} g\in \Ho^{1,2}(\Omega,\Omega^*)\}$ and $\tau_\diamond=\sup
\mathcal{T}$. Then the function $c=c(\tau)$ is a strictly decreasing
function in $\mathcal{T}$.
\end{corollary}
\begin{proof}  Suppose that Let $0<\tau_\circ, \tau<\tau_\diamond$ and $\tau_\circ, \tau\in\mathcal{T}$. Assume
that $h\in \D^\rho(A(\tau_\circ),\Omega^\ast)$ and $g\in
\D^\rho(A(\tau),\Omega^\ast)$ are diffeomorphic energy-minimal
deformations. Assume that $\tau<\tau_\circ$. Prove that
$c(\tau_\circ)<c(\tau)$. From \eqref{claim100} we have
$${\E^\rho}[g]-{\E^\rho}[h] \ge 8\pi c(\tau_\circ) (\tau_\circ-\tau)$$ and
$${\E^\rho}[h]-{\E^\rho}[g] \ge 8\pi c(\tau)
(\tau-\tau_\circ).$$ Dividing by $\tau_\circ-\tau$ the previous
inequalities we arrive at inequality $c(\tau_\circ)\le c(\tau)$.
Because $A(\tau_\circ)$ and $A(\tau)$ are not  conformally
equivalent it follows that $c(\tau_\circ)< c(\tau)$  as desired.
\end{proof}
\begin{remark}\label{bela}
Let $\tau_n \in \{\tau: \EE^\rho(\tau,\omega)=\E^\rho[g]\text{ for
some } g\in \Ho^{1,2}(A(\tau),\omega)\}$ be a sequence converging to
$\tau_\diamond$ and let $g_n$ be a sequence of harmonic
diffeomorphisms such that $\EE^\rho(\tau_n,\omega)=\E^\rho[g_n]$.
Then $g_n$, up to some subsequence, converges to a harmonic
diffeomorphism $h_\diamond\in \Ho^{1,2}(A(\tau_\diamond),\omega)$.
It is called the critical $\rho$- Nitsche map and $A(\tau_\diamond)$
is called the critical domain for $\Omega^*$.  The explicit
evaluation of $\tau_\diamond$ is not known for arbitrary $\Omega^*$
i.e. for arbitrary $\rho$. However for circular annuli and radial
metrics (in particluar for Euclidean metric) the constant
$\tau_\diamond$ is calculated exactly (see \cite{AIM} and
\cite{dist}). The strict convexity part of
Theorem~\ref{convextheorem} fails for $\tau>\tau_\diamond$. We
demonstrate this with Theorem~\ref{circ} based on the results
of~\cite{dist}.
\end{remark}
\section{Convexity of  minimal energy for radial metrics}\label{K}
\begin{definition}\label{define} The radial
metric $\rho$ is called a \emph{regular metric} if $$\inf_{\delta<
s< \sigma}s\rho(s)=\lim_{s\to \delta+0}s\rho(s)$$ and has bounded
Gauss curvature $K$.
\end{definition}

From now on we will assume that the metric $\rho$ is regular in the
sense of Definition~\ref{define}. We recall first some results from
\cite{dist}. Namely in \cite{dist} are  found all examples  $w$ of radial
$\rho$-harmonic maps between annuli. The mapping $w$ given by
\begin{equation}\label{w}w(se^{it}) =
q^{-1}(s)e^{it},\end{equation} where
\begin{equation}\label{var}q(s) =\exp\left(\int_{\sigma}^s
\frac{dy}{\sqrt{y^2+\gamma\varrho^2 }}\right), \ \delta\le s\le
\sigma,
\end{equation} and $\gamma$ satisfies the condition:
\begin{equation}\label{unt} y^2+\gamma\varrho^2(y)\ge 0, \;\text{for} \; \delta\le s\le \sigma
,\end{equation} is a $\rho$-harmonic mapping between annuli
$A=A(r,1)$ and $A'=A(\delta,\sigma)$, where
\begin{equation}\label{rr}r=\exp\left(\int_{\sigma}^\delta \frac{dy}{\sqrt{y^2+\gamma\varrho^2
}}\right).\end{equation} The harmonic mapping $w$ is normalized by
$$w( e^{it})=\sigma e^{it}.$$ The mapping $w=h^\gamma(z)$ is a diffeomorphism, and is called  \emph{$\rho$-Nitsche map}.

Then \eqref{unt} is equivalent to
\begin{equation}\label{unti}\delta^2+\gamma\varrho^2(\delta)\ge
0.\end{equation} Accordingly, for $\gamma=-\delta^2\rho^2(\delta)$,
we have well defined function
\begin{equation}\label{qu}q_{\diamond}(s)=\exp\left(\int_{\sigma}^s
\frac{dy}{\sqrt{y^2-\delta^2\rho^2(\delta)\varrho^2 }}\right),
\delta\le s\le \sigma.\end{equation} The mapping $h_{\diamond}:A\to
A'$ defined by $h_{\diamond}(se^{it})=q_{\diamond}^{-1}(s)e^{it}$ is
called the \emph{critical Nitsche map}.

Notice that, the mapping $$f^{\gamma}(se^{it})=q(s)e^{it}:A'\to A$$
is the inverse of the harmonic diffeomorphism $w$.
\begin{conjecture}\label{onc}\cite{dist} {\it Let $\rho$ be a regular metric. 
 If $r<1$, and there exists a $\rho-$ harmonic mapping of the annulus
$A'=A(r,1)$ onto the annulus $A=A(\delta,\sigma)$, then
\begin{equation}\label{nit}r\ge \exp\left(\int_{\sigma}^{\delta}
\frac{\rho(y)dy}{\sqrt{y^2\rho^2(y)-\delta^2\rho^2(\delta)
}}\right).\end{equation}}
\end{conjecture}
{ Notice that if $\rho =1$, then this conjecture coincides with
standard Nitsche conjecture and the latter conjecture is settled
recently by Iwaniec, Kovalev and Onninen in \cite{IKO2}.} For a
regular metric we refer \cite{kalaj} for some partial solution of
Conjecture~\ref{onc}.

In \cite{dist} it is proved that every Nitsche map $w(z)=p(r)e^{it}$
is a minimizing deformation and consequently a stationary
deformation. In the following lemma we demonstrate the validity of
Lemma~\ref{ctheory} to this class of mappings.
\begin{lemma}
For every Nitsche map $w=h^\gamma(z)=p(r)e^{it}$, where $z=r e^{it}$
and $q(r)=p^{-1}(r)$ we have
\begin{equation}\label{hop}\mathrm{Hopf}(w)=\frac{\gamma}{4 z^2}.\end{equation}
In the notation of previous sections we have
\begin{equation}\label{gac}\gamma=4c.\end{equation}
\end{lemma}

\begin{proof}
 Straightforward calculations yield that
\begin{equation}\label{q}w_z \overline{w_{\bar z}}
=\frac{r^2\left(p'(r)\right)^2- p(r)^2}{4z^2}.\end{equation} Further
by differentiating
$$r=\exp\left(\int_{\sigma}^{p(r)}\frac{d y}{\sqrt{y^2+\gamma\varrho^2(y)
}}\right)$$ we obtain
$$\frac{1}{r^2}=\left(p'(r)\right)^2\frac{\rho^2(p(r))}{\gamma+p^2(r)
\rho^2(p(r))}.$$ Since
\begin{equation}\label{hopf}\mathrm{Hopf}(w)=\rho^2(w(z)) w_z
\overline {w_{\bar z}},\end{equation} we obtain \eqref{hop}.
\end{proof}
\begin{theorem}\label{circ} Let $\Omega=A(r,1)$ and $\Omega^*=A(R,1)$ where $r<1$ and
$R<1$. Let  $r=e^{-\tau}$ and $\omega=e^{-R}$. Then $\tau =
\Mod\Omega$ and $\omega=\Mod \Omega^*$. Let
\begin{equation}\label{Psi}\Psi_\rho(\omega):= \int_{R}^{1}
\frac{\rho(y)dy}{\sqrt{y^2\rho^2(y)-R^2\rho^2(R) }}.\end{equation}
The function $\tau\mapsto \EE^\rho(\tau,\omega)$ is $\CC^2$-smooth
on $(0,\infty)$, strictly convex for
\begin{equation*}\tau\le \Psi_\rho(\omega) \end{equation*} and affine for
\begin{equation*}\tau> \Psi_\rho(\omega).\end{equation*}
\end{theorem}
Notice that in this case we have $$\tau_\diamond=\Psi_\rho(\Mod
\Omega^*)=\int_{R}^{1}
\frac{\rho(y)dy}{\sqrt{y^2\rho^2(y)-R^2\rho^2(R) }}>
\log\frac{1}{R}=\Mod(\Omega^*).$$ Before we proceed to the proof we
recall that \cite[Example~10.3]{inve} contains a special case for
$\rho=1$. The proof presented here is different and we believe it
can be of more interest because it suggest the Conjecture~\ref{con}.
\begin{proof}
 Assume that there is a
$\rho-$harmonic mapping between $\Omega$ and $\Omega^*=A(R,1)$. We
begin with the case
\begin{equation*}r\ge \exp\left(\int_{1}^{R}
\frac{\rho(y)dy}{\sqrt{y^2\rho^2(y)-R^2\rho^2(R)
}}\right).\end{equation*} By~\cite[Corollary~3.4]{dist} the absolute
minimum of the energy integral $$h\to E_\rho[h],\ \ h\in
W^{1,2}(A,A')$$ is attained by a $\rho-$Nitsche map
$$h^{\gamma}(z)=q^{-1}(s)e^{i(t+\beta)}, \; z=se^{it}, \
\beta\in[0,2\pi),$$ where $$q(s) =\exp\left(\int_{1}^s
\frac{dy}{\sqrt{y^2+\gamma\varrho^2 }}\right), R<s<1,$$ where
$\gamma=\gamma(r)$ is defined by
\begin{equation}\label{rc}r=\exp\left(\int_{1}^R
\frac{dy}{\sqrt{y^2+\gamma\varrho^2 }}\right).\end{equation} Further
for $p(x)=q^{-1}(x)$ we compute
\begin{equation}\label{circ2}
{\E^\rho}[h^{\gamma}] =2\pi
\int_{e^{-\tau}}^1\frac{\gamma+2\rho^2(p(t))p^2(t)}{t} dt,
\end{equation}
which  yields
\begin{equation}\label{circ21}
{\E^\rho}[h^{\gamma}] =2\pi
\int_{R}^1\left(\gamma(r)+2\rho^2(s)s^2\right)\frac{q'(s)}{q(s)} ds.
\end{equation}
Since
$$\frac{q'(s)}{q(s)}=\frac{1}{\sqrt{s^2+\gamma(r)\varrho^2(s)}}$$
we obtain that

\begin{equation}\label{circ22}
{\E^\rho}[h^{\gamma}]=2\pi
\int_{R}^1\frac{\gamma(r)+2\rho^2(s)s^2}{\sqrt{s^2+\gamma(r)\varrho^2(s)}}
ds\end{equation} which  implies
\begin{equation}\label{circ3}\epsilon(r):=
\EE^\rho(\log 1/r,\Omega^*) = 2\pi
\int_{R}^1\frac{\gamma(r)+2\rho^2(s)s^2}{\sqrt{s^2+\gamma(r)\varrho^2(s)}}ds.
\end{equation}
It follows from \eqref{rc} that $\gamma$ is a differentiable
function. By differentiating \eqref{circ3} with respect to $r$ we
get
\begin{equation}\frac{ d\epsilon(r)
}{dr}=\pi
\gamma(r)\gamma'(r)\int_R^1\frac{\varrho^2(s)ds}{\left(s^2+\gamma(r)\varrho^2(s)\right)^{3/2}}.\end{equation}
But
\begin{equation}\label{fun}\gamma'(r)\int_R^1\frac{\varrho^2(s)ds}{\left(s^2+\gamma(r)\varrho^2(s)\right)^{3/2}}=\frac{2}{r}\end{equation}
 and therefore
\begin{equation}\frac{ d\epsilon(r)
}{dr}=\frac{2\pi \gamma(r)}{r}.\end{equation} Thus
$$\frac{d\EE^\rho(\tau,\omega)}{d\tau} =-r\frac{2\pi
\gamma(r)}{r}$$ i.e. $$\frac{d\EE^\rho(\tau,\omega)}{d\tau} =-2\pi
\gamma(r).$$ Further
\begin{equation}\label{eq}\frac{d^2\EE^\rho(\tau,\omega)}{d\tau^2} =-r(-2\pi
\gamma'(r))=2r\pi \gamma'(r).\end{equation} Since $\gamma'(r)>0$ for
all $r$ it follows that $\EE^\rho(\tau,\omega)$ is strictly convex
under $\rho-$Nitsche condition. Observe that for $R=r$, $c(r)=0$,
and therefore
\begin{equation}\label{eps}\frac{d\EE^\rho(\tau,\omega)}{d\tau} =0.\end{equation}
Further for $$r_\diamond= \exp(\tau_\diamond)=\exp\left(\int_{1}^{R}
\frac{\rho(y)dy}{\sqrt{y^2\rho^2(y)-R^2\rho^2(R) }}\right),$$ i.e.
\begin{equation}\label{gdiamond}\gamma_\diamond:=\gamma(r_\diamond)=-R^2\rho^2(R)\end{equation} we have  \begin{equation}\label{forp}\frac{d\EE^\rho(\tau,\omega)}{d\tau} =2\pi
R^2\rho^2(R)\end{equation} and
\[\begin{split}\lim_{\tau\to \tau_\diamond-0}\frac{d^2\EE^\rho(\tau,\omega)}{d\tau^2}
&=\lim_{r\to r_\diamond}\pi \gamma'(r)\\&=\lim_{r\to r_\diamond}\pi
\left(\int_R^1\frac{\varrho^2(s)ds}{\left(s^2+\gamma(r)\varrho^2(s)\right)^{3/2}}\right)^{-1}/(2r)\\&=
\left(\int_R^1\frac{\varrho^2(s)ds}{\left(s^2-R^2\rho^2(R)\varrho^2(s)\right)^{3/2}}\right)^{-1}/(2r_\diamond)=0,\end{split}\]
i.e. \begin{equation}\label{fori}
\frac{d^2\EE^\rho(\tau,\omega)}{d\tau^2}|_{\tau=\tau_\diamond-0}=0.
\end{equation}
It remains to consider the case \begin{equation*}r<
\exp\left(\int_{1}^{R}
\frac{\rho(y)dy}{\sqrt{y^2\rho^2(y)-R^2\rho^2(R)
}}\right).\end{equation*} Let
$$r_\diamond=\exp\left(\int_{1}^{R}
\frac{\rho(y)dy}{\sqrt{y^2\rho^2(y)-R^2\rho^2(R) }}\right).$$ Then
$r<r_\diamond<1$.  By~\cite[\S~ 4.2.]{dist} the infimum
$\EE^\rho(\Omega,\Omega^*)$ is realized by a non-injective
deformation $h \colon \Omega \onto \Omega^\ast$
\[h(z)=\begin{cases} R\frac{z}{\abs{z}} & \mbox{for } r<\abs{z} \le r_\diamond\\
h_{\diamond}(z)& \mbox{ for } r_\diamond\le \abs{z} <1
\end{cases}\] where $h_{\diamond}(z)=p_{\diamond}(|z|)e^{it}=(q_{\diamond})^{-1}(|z|)e^{it}$, $z=|z|e^{it}$ and  $q_{\diamond}$ is defined in \eqref{qu}. Here the radial projection $z\mapsto R z/\abs{z}$
hammers $A(r,r_\diamond)$ onto the circle $|z|=R$ while the critical
$\rho-$Nitsche mapping $h_{\diamond}$ takes $A(r_\diamond,1)$
homeomorphically onto $\Omega^\ast$. The contribution of the radial
projection to the energy of $h$ is equal to $$2\pi
\int_{r}^{r_\diamond} \rho^2(R)R^2\frac{dx}{x}$$ which together with
contribution of $h_{\diamond}$ gives
\begin{equation}\label{cont}
\EE^\rho(\Omega,\Omega^*) = -2\pi\gamma \int_{r}^{r_\diamond}
\frac{dx}{x}+2\pi
\int_{r_\diamond}^1\frac{\gamma+2\rho^2({p_{\diamond}}(t))p_{\diamond}^2(t)}{t}
dt,
\end{equation}
where $\gamma=-\rho^2(R)R^2$. This is an affine function of
$\Mod(\Omega)=\log 1/r$ whose first derivative equals $2\pi
\rho^2(R)R^2$ and the second derivative vanishes. From \eqref{cont}
we have
$$\frac{d\EE^\rho(\tau,\omega)}{d\tau}|_{\tau=\log\frac{1}{r_\diamond}}=2\pi\rho^2(R )R^2=-2\gamma=-8c.$$ Thus $\EE^\rho(\tau,\omega)$ is a
$\CC^2$-smooth function in $(0,\infty)$.
\end{proof}
\begin{lemma}\label{motiv} The harmonic mapping $f^\gamma(z)=p^\gamma(s)e^{it}$ defined in \eqref{w} is
quasiconformal minimizer if and only if $\gamma>-R^2\rho^2(R)$ or
what is the same if and only if $\Mod(A(r,1))<\tau_\diamond$, where
$\tau_\diamond$ is defined in Theorem~\ref{circ}. If
$\gamma=-R^2\rho^2(R)$, then the minimizer is harmonic but is not
quasiconformal. The constant of quasiconformality is
$$K_{f^\gamma}=\max\left\{\sqrt{\frac{\gamma_\diamond
-\gamma}{\gamma_\diamond}},\sqrt{\frac{\gamma_\diamond}{\gamma_\diamond
-\gamma}}\right\}.$$ Here $ \gamma_\diamond=-R^2\rho^2(R)$ (see \eqref{gdiamond}).
\end{lemma}
\begin{proof} First of all $$\mu_f(z)=\frac{|s p'(s)-p(s)|}{ sp'(s)+p(s)}.$$
From
$$\frac{1}{s^2}=\left(p'(r)\right)^2\frac{\rho^2(p(r))}{\gamma+p^2(s)
\rho^2(p(s))},$$ we obtain
$$\mu_f(z)=\left|\frac{-p(s)\rho(p(s))+\sqrt{\gamma+p^2(s)\rho^2(p(s))}}{p(s)\rho(p(s))+\sqrt{\gamma+p^2(s)\rho^2(p(s))}}\right|,$$
and therefore $\mu_f(z)<1$ if and only if
$\gamma+p^2(s)\rho^2(p(s))>0$ for $R\le s\le 1$ and therefore for
$\gamma>-R^2\rho^2(R)$. Moreover
$$K_f(z)=\frac{1+\mu_f(z)}{1-\mu_f(z)}=
\max\left\{\frac{\sqrt{\gamma + p^2(s) \rho^2(p(s))}}{p(s)
\rho(p(s))},\frac{p(s) \rho(p(s))}{\sqrt{\gamma + p^2(s)
\rho^2(p(s))}}\right\}$$ and thus $$K_f:=\max_{z}K_f(z)=
\max\left\{\frac{\sqrt{\gamma + p^2(R) \rho^2(p(R))}}{p(R)
\rho(p(R))},\frac{p(R) \rho(p(R))}{\sqrt{\gamma + p^2(R)
\rho^2(p(R))}}\right\},$$ which can be written as
$$K_f=\max\left\{\sqrt{\frac{\gamma_\diamond
-\gamma}{\gamma_\diamond}},\sqrt{\frac{\gamma_\diamond}{\gamma_\diamond
-\gamma}}\right\},$$ where $\gamma=4c(\tau)$ and
$\gamma_\diamond=4c(\tau_\diamond)$.  Thus the minimum of $K_f$ is
for $\gamma=0$, i.e. for $\tau=\Mod(\Omega)=\Mod(\Omega^*)$ and is
increasing for $\gamma>0$ i.e. $\tau<\Mod(\Omega^*)$ and decreasing
for $\gamma<0$ i.e. $\tau>\Mod(\Omega^*)$. This shows that $K_f$
depends only on $\Mod(\Omega)$ and $\Mod(\Omega^\diamond)$.
\end{proof}
Lemma~\ref{motiv} is a motivation for the following conjecture
\begin{conjecture}\label{con}
a) If $f$ is a minimizer of the energy $\E^\rho$ between two doubly
connected domains $\Omega$, ($\tau=\Mod(\Omega)$) and $\Omega^*$,
then $f$ is harmonic and $K(\tau)$-quasiconformal if and only if
$\tau$ is smaller that the modulus $\tau_\diamond$ of critical
Nitsche domain $A(\tau_\diamond)$ (see Remark~\ref{bela} and
Corollary~\ref{newa}).

b) Under condition of a) we conjecture that
$\gamma_\diamond=4c_\diamond<0$ (see Remark~\ref{bela}) and
\begin{equation}\label{kon}K(\tau)=\max\left\{\sqrt{\frac{\gamma_\diamond
-\gamma}{\gamma_\diamond}},\sqrt{\frac{\gamma_\diamond}{\gamma_\diamond
-\gamma}}\right\},\end{equation} where $\gamma=4c(\tau)$ and
$\gamma_\diamond=4c(\tau_\diamond)$.
\end{conjecture}
Concerning Conjecture~\ref{con} we offer the following observation
\begin{remark} a) As every minimizer is unique up to conformal transformation of the domain, provided that it is a
homeomorphism, it follows from Lemma~\ref{newa} that
Conjecture~\ref{con} holds true if the image domain is a circular
annulus and $\rho$ is a radial metric. In this case the solution of
minimization problem for $L^1$-norm of distortion between circular
annuli is the quasiconformal mapping $f_1(z)=q^\gamma(|z|)e^{i\arg
z}$. If instead of $L^1$-norm of distortion we consider
$L^\infty$-norm of distortion, then the minimum is attained for the
mapping $f_\infty(z)=|z|^{\log r/\log R}e^{i\arg z}$. As $f_1$
and $f_\infty$ are quasiconformal mappings between annuli $A(R,1)$
and $A(r,1)$, it follows from quasiconformal theory that
$$K_{f_\infty}=\max\left\{\frac{\log R}{\log r},\frac{\log r}{\log
R}\right\}\le K_{f_1}$$ with equality if and only if $\rho(z)=1/|z|$
or $\gamma=0$.

b) It follows from \eqref{hopf5} that for every stationary
transformation $h$ we have
$$\mathbb
K_h:=\frac{1}{2}\left(K_h+\frac{1}{K_h}\right)=\frac{\|D
h\|^2}{2J_h}=\frac{|h_N|^2+|h_T|^2}{2|h_N|\cdot |h_T|}.$$
 Combining with \eqref{hopf2a} we obtain $$\mathbb
K_h=\frac{2|h_N|^2-\frac{4c}{\rho^2(h(z))|z|^2}}{2\sqrt{|h_N|^2(|h_N|^2-\frac{4c}{\rho^2(h(z))|z|^2})}}.$$
Therefore for $\Gamma=\rho^2(h)|h_N|^2|z|^2$ and $\gamma=4c$
$$K_h(z)=\max\left\{\sqrt{\frac{\Gamma-\gamma}{\Gamma}},\sqrt{\frac{\Gamma}{\Gamma-\gamma}}\right\}.$$
As $c(\tau_\diamond)\le  0$, it can be shown that $K_h(z)\le K(\tau)$
if and only if $$\rho^2(h)|h_N|^2|z|^2\ge
4(c(\tau)-c(\tau_\diamond))$$ or what is the same if and only if
$$\rho^2(h)|h_T|^2|z|^2\ge -4c(\tau_\diamond).$$ We believe that the
inequality \eqref{hilb} above can be used in this context.

 c) It  has been shown in some
recent papers that every q.c. harmonic mapping between domains with
smooth boundaries, say $C^{1,\alpha}$ is bi-Lipschitz continuous
(see e.g. \cite{trans}-\cite{kalann}, \cite{MP}, \cite{MMV}). If
Conjecture~\ref{con} is true, then every minimizer between smooth
double connected domains is a bi-Lipschitz mapping provided that
$\tau<\tau_\diamond$.
\end{remark}
\subsection{Appendix} Here we give two important metrics for which the results of this section can be stated in a more explicit way.
\begin{example} Let  $\rho$ be the Riemann metric
$\rho=\frac{2}{1+|z|^2}.$ Equation \eqref{el} becomes
\begin{equation}\label{rimequ}u_{z\bar
z}-\frac{2\bar u}{1+|u|^2}u_z\cdot u_{\bar z}=0.\end{equation}

Notice this important example. The Gauss map of a surface $\Sigma$
in $\Bbb R^3$ sends a point on the surface to the corresponding unit
normal vector $\mathbf{n}\in \overline{\Bbb C} \cong S^2$. In terms
of a conformal coordinate $z$ on the surface, if the surface has
{\it constant mean curvature}, its Gauss map $\mathbf{n}: \Sigma
\mapsto \overline{\Bbb C}$, is a Riemann harmonic map \cite{rv}.
Since $$\int_{{\Bbb C}}\rho^2(w)dudv=4\pi,$$it follows that the
Riemann metric is allowable for every double connected domain
(bounded or unbounded).
\end{example}
\begin{example}
If $u:\Bbb U\mapsto \Bbb U$ is a harmonic mapping with respect to
the hyperbolic metric $\lambda=\dfrac{2}{1-|z|^2}$ then
Euler-Lagrange equation of $u$ is
\begin{equation}\label{eqpo}
u_{z\bar z}+\frac{2\bar u}{1-|u|^2}u_z\cdot u_{\bar z}=0.
\end{equation}
An important example of hyperbolic harmonic mapping is the Gauss map
of a space-like surfaces with constant mean curvature $H$ in the
Minkowski $3$-space $M^{2,1}$ (see \cite{ctr}). This metric is
allowable in compact bounded domains in $\mathbb U$ but for every
$r<1$, the integral $\int_{A(r,1)}\lambda^2(w)dudv$ diverges.
\end{example}
\subsection*{Acknowledgement} I thank Professor Leonid Kovalev for very useful discussion about the subject of this paper.
\bibliographystyle{amsplain}

\end{document}